\documentclass{article}

\newcommand{\Debug}{1}

\usepackage{amsmath, amsthm, amssymb, amstext, epsf,amsfonts}
\usepackage{enumerate}
\usepackage{graphicx}
\usepackage[applemac]{inputenc}

\usepackage[percent]{overpic}

\theoremstyle{plain}

\newtheorem{thm}{Theorem}[section]
\newtheorem{lem}[thm]{Lemma}
\newtheorem{cor}[thm]{Corollary}

\newtheorem{conj}[thm]{{Conjecture}}

\newtheorem{problem}[thm]{{Problem}}

\theoremstyle{definition}
\newtheorem{defi}[thm]{Definition}

\newcommand{\Cg}{Cayley graph}

\newcommand{\R}{\ensuremath{\mathbb{R}}}
\newcommand{\N}{\ensuremath{\mathbb{N}}}
\newcommand{\Z}{\ensuremath{\mathbb{Z}}}

\newcommand{\sm}{\ensuremath{\setminus}}

\newcommand{\inv}{\ensuremath{^{-1}}}

\newcommand{\rand}{\partial}
\newcommand{\Aut}{\textnormal{Aut}}

\newcommand{\sub}{\subseteq}

\newcommand{\comment}[1]{}

\newcommand{\nat}{{\mathbb N}}

\newcommand{\BF}{\ensuremath{\mathcal B}}
\newcommand{\CF}{\ensuremath{\mathcal C}}
\newcommand{\DF}{\ensuremath{\mathcal D}}
\newcommand{\EF}{\ensuremath{\mathcal E}}
\newcommand{\FF}{\ensuremath{\mathcal F}}

\newcommand{\HF}{\ensuremath{\mathcal H}}
\newcommand{\IF}{\ensuremath{\mathcal I}}

\newcommand{\PF}{\ensuremath{\mathcal P}}

\newcommand{\RF}{\ensuremath{\mathcal R}}
\newcommand{\SF}{\ensuremath{\mathcal S}}
\newcommand{\TF}{\ensuremath{\mathcal T}}

\newcommand{\VF}{\ensuremath{\mathcal V}}
\newcommand{\WF}{\ensuremath{\mathcal W}}

\newcommand{\prbl}{pre-block}
\newcommand{\prcl}{pre-cluster}

\newcommand{\blk}{block}
\newcommand{\nss}{non-self-separating}
\newcommand{\sese}{self-separating}
\newcommand{\splpr}{special \plpr}
\newcommand{\glplpr}{general \plpr}
\newcommand{\gplpr}{general \plpr}
\newcommand{\gcplpr}{generic \plpr}
\newcommand{\gempr}{generic \empr}

\newcommand{\Kc}{\ensuremath{\overline{K}}}
\newcommand{\cellsk}{\ensuremath{K^2}}
\newcommand{\pii}{\ensuremath{\pi^{-1}}}
\newcommand{\suf}{super-face}
\newcommand{\wrt}{with respect to}
\newcommand{\utr}{up to reflection}
\newcommand{\xl}{\ensuremath{X^\ell}}

\ifnum \Debug = 1 
	\def\?#1{\vadjust{\vbox to 0pt{\vss\vskip-8pt\leftline{%
     \llap{\hbox{\vbox{\pretolerance=-1
     \doublehyphendemerits=0\finalhyphendemerits=0
     \hsize16truemm\tolerance=10000\small
     \lineskip=0pt\lineskiplimit=0pt
     \rightskip=0pt plus16truemm\baselineskip8pt\noindent
     \hskip0pt        
     #1\endgraf}\hskip7truemm}}}\vss}}}
\else \newcommand{\?}[1]{} \fi

\begin{document}

\title{The planar Cayley graphs are effectively enumerable II}
\author{Agelos Georgakopoulos\thanks{Supported by EPSRC grant EP/L002787/1, and by the European Research Council (ERC) under the European Union's Horizon 2020 research and innovation programme (grant agreement No 639046). The first author would like to thank the Isaac Newton Institute for Mathematical Sciences, Cambridge, for support and hospitality during the programme `Random Geometry' where work on this paper was undertaken.}
\medskip 
\\
  {Mathematics Institute}\\
 {University of Warwick}\\
  {CV4 7AL, UK}\\
\and
 Matthias Hamann\thanks{Supported by the Heisenberg-Programme of the Deutsche Forschungsgemeinschaft (DFG Grant HA 8257/1-1).\newline Both authors have been supported by FWF grant P-19115-N18.} 
  \medskip \\
{Alfr\'ed R\'enyi Institute of Mathematics}\\
{Hungarian Academy of Sciences}\\
{Budapest, Hungary}\\
\\
}
\date{\today}
\maketitle

\newcommand{\labtequ}[2]{ \begin{equation} \label{#1} 	\begin{minipage}[c]{0.9\textwidth} \centering #2 \end{minipage} \ignorespacesafterend \end{equation} } 
\newcommand{\showFig}[2]{
   \begin{figure}[htbp]
   \centering
   \noindent
\includegraphics{#1.eps}
   \caption{\small #2}
   \label{#1}
   \end{figure}
}
\newcommand{\fig}[1]{Figure~{\ref{#1}}}

\newcommand{\cR}{\ensuremath{\mathcal{R}}}
\newcommand{\srev}{spin-reversing}
\newcommand{\plpr}{planar presentation}
\newcommand{\empr}{embedded presentation}

\newcommand{\Plpr}{Planar presentation}
\newcommand{\spre}{spin-preserving}
\newcommand{\Gam}{\ensuremath{\Gamma}}
\newcommand{\sig}{\ensuremath{\sigma}}
\newcommand{\g}{\ensuremath{G\ }}
\newcommand{\G}{\ensuremath{G}}
\renewcommand{\iff}{if and only if}
\newcommand{\fe}{for every}
\newcommand{\Fe}{For every}
\newcommand{\fea}{for each}
\newcommand{\Fea}{For each}
\newcommand{\st}{such that}
\newcommand{\sot}{so that}
\newcommand{\ti}{there is}
\newcommand{\ta}{there are}

\newcommand{\cns}{consistent}
\newcommand{\cnsem}{consistent embedding}
\newcommand{\vapf}{accumulation-free}
\newcommand{\dray}{double ray}
\newcommand{\cls}[1]{\overline{#1}}


\newcommand{\Lr}[1]{Lemma~\ref{#1}}
\newcommand{\Lrs}[1]{Lemmas~\ref{#1}}
\newcommand{\Tr}[1]{Theorem~\ref{#1}}
\newcommand{\Trs}[1]{Theorems~\ref{#1}}
\newcommand{\Sr}[1]{Section~\ref{#1}}
\newcommand{\Srs}[1]{Sections~\ref{#1}}
\newcommand{\Prr}[1]{Pro\-position~\ref{#1}}
\newcommand{\Prb}[1]{Problem~\ref{#1}}
\newcommand{\Cr}[1]{Corollary~\ref{#1}}
\newcommand{\Cnr}[1]{Con\-jecture~\ref{#1}}
\newcommand{\Or}[1]{Observation~\ref{#1}}
\newcommand{\Er}[1]{Example~\ref{#1}}
\newcommand{\Dr}[1]{De\-fi\-nition~\ref{#1}}

\newcommand{\kreis}[1]{\mathaccent"7017\relax #1}

\newcommand{\ttt}{\ensuremath{\mathbb T}}
\newcommand{\sydi}{\triangle}
\newcommand{\Gc}{\ensuremath{\overline{G}}}
\newcommand{\cells}{\ensuremath{G^2}}
\newcommand{\BS}{\ensuremath{\mathbb S}}
\renewcommand{\PF}{\ensuremath{\mathcal P}}

\begin{abstract}
We show that a group admits a planar, finitely generated \Cg\ if and only if it admits a special kind of group presentation we introduce, called a planar presentation. Planar presentations can be recognised algorithmically. As a consequence, we obtain an effective enumeration of the planar Cayley graphs, yielding in particular an affirmative answer to a question of Droms et al.\ asking whether the planar groups can be effectively enumerated. 
\end{abstract}

\section{Introduction}

In this paper we complete an effort, started in \cite{planarpresI}, and building upon \cite{cayley3,cay2con}, the aim of which is to understand the planar \Cg s. In \cite{planarpresI} we handled the special case of 3-connected \Cg s, and more generally, \Cg s $G$ that admit a {\em consistent} embedding in $\R^2$, that is, an embedding the facial paths of which are preserved by the action on $G$ by its group (see \Sr{secDem} for a more detailed definition).
It is shown in a follow-up paper in preparation \cite{fpd} that, at least in the finitely generated case, the groups having such \Cg s  are exactly those  groups admitting a faithful, properly discontinuous action by homeomorphisms on a 2-manifold contained in  $\BS^2$. It is  shown in the same paper that there is a planar \Cg\ the group of which cannot act faithfully and properly discontinuously on $\BS^2$. Therefore, the aforementioned groups form a proper subclass of the \emph{planar groups}, i.e.\ the groups admitting a planar \Cg. In this paper we broaden the group presentations introduced in  \cite{planarpresI} so that we capture exactly the planar finitely generated \Cg s. In particular, we capture the planar groups, and we show that they can be effectively enumerated, answering a question of Droms et al.\ \cite{DroInf,DrSeSeCon}. 

\medskip

The \emph{Cayley complex} $X$ corresponding to a group presentation $\PF=\left< \SF \mid \mathcal \RF \right>$ is the 2-complex obtained from the \Cg\ $G$ of $\PF$ by glueing a 2-cell along each closed walk of $G$ induced by a relator $R\in \RF$. We say that $X$ is \emph{almost planar}, if it admits a map $\rho: X \to \R^2$ \st\ the 2-simplices of $X$ are \emph{nested} in the following sense. We say that  two 2-simplices of $X$ are nested, if the images of their interiors are either disjoint, or one is contained in the other, or their intersection is the image of a 2-cell bounded by two parallel edges corresponding to an involution $s\in \SF$.\footnote{The third option can be dropped by considering the \emph{modified} Cayley complex in the sense of \cite{LyndonSchupp}, i.e.\ by representing involutions in $\SF$ by single, undirected edges.} We call the presentation $\PF$ a \emph{\plpr} if its Cayley complex is almost planar. We will show that every planar, finitely generated \Cg\ admits a \plpr.
However, we prove something much stronger than that. We are going to introduce a specific type of \plpr, called a \emph{\gplpr}, and show that every planar, finitely generated \Cg\ admits such a
presentation and, conversely, every \gplpr\ has a planar \Cg\ (\Tr{Gplanar}). This converse is the hardest result of this paper. Our main result is:
\begin{thm} \label{mainthm}
A finitely generated group admits a planar \Cg\ if and only if it admits a \glplpr. 
\end{thm}

The main idea of its proof is that if two relators in a presentation induce cycles whose interiors overlap but are not nested (in a sense similar to the nestedness of 2-simplices), then we replace a subword of one relator by a subword of the other to produce an equivalent presentation with less overlapping; our proof that a presentation with no such overlaps exists is based on a dual version of the machinery of Dunwoody cuts~\cite{dicks_dunw}, but for cycles instead of cuts.

As a corollary of Theorem~\ref{mainthm} we obtain that the planar \Cg s, and hence their groups, can be effectively enumerated (\Tr{thm_maincon2}). This answers a question of Droms et al.\ \cite{DroInf,DrSeSeCon}. M.~Dunwoody (private communication) informs us that the fact that the planar groups  can be effectively enumerated should also follow from his result \cite[Theorem~3.8]{dunPla} with a little bit of additional work (the main issue here is whether the `or a subgroup of index two' proviso can be dropped).

For more on the motivation of this work and the general background we refer the reader to  \cite{planarpresI}. We note that while it may help the reader to know~\cite{planarpresI}, the present paper is self-contained.

\subsection{\Plpr s}
The formal definition of a \glplpr\ is given in \Sr{secConc} and it is a slight generalisation of the notion of a \emph{generic} \plpr\  defined in \Sr{secpltygen}. Here, we are going to sketch the most interesting special case of this concept, called a \splpr. Such presentations always exist for a 3-connected planar \Cg, or more generally, for a \Cg\ that can be embedded in the plane in such a way that its label-preserving automorphisms carry facial paths to facial paths.

We say that $\PF=\left< \SF \mid \mathcal \RF \right>$ is a \emph{\splpr}, if it can be endowed with a cyclic ordering $\sig$ ---from now on called a \emph{spin}--- of the symmetrization $\SF'=\{s,s^{-1} \mid s \in \SF\}$ of its generating set, with the following property.
Suppose $W_1=sUt$ and $W_2=s'Ut'$, where $s,s',t,t'\in \SF'$, are two words, each contained in some rotation of a relator in $\RF$  (possibly the same relator), where $U$ is any (possibly trivial) word with letters in $\SF'$. Then $\sig$ allows us to say whether paths induced by these words $W_1,W_2$ would \emph{cross} each other or not if we could embed the \Cg\ of $\PF$ in the plane in such a way that for every vertex the cyclic ordering of the labels of its incident edges we observe coincides with $\sig$. To make this more precise, we embed a tree consisting of a `middle' path $P$ with edges labelled by the letters in $U$, and two leaves attached at each endvertex of~$P$ labelled with $s,s',t,t'$ as in \fig{stu}, where the spin we use at each endvertex of~$P$ is the one induced by $\sig$ on the corresponding 3-element subset of $\SF'$. There are essentially two situations that can arise, both shown in that figure. Naturally, we say that $W_1,W_2$ \emph{cross} each other in the right-hand situation, and they do not in the left-hand one.

\showFig{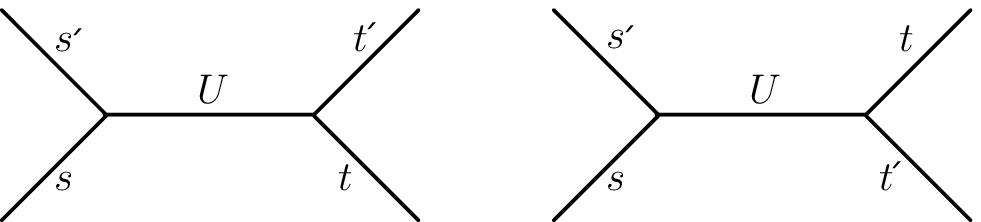}{The definition of \emph{crossing}; $W_1=sUt$ crosses $W_2=s'Ut'$ in the right, but not in the left.}

We then say that $\PF$ is a \splpr, if there is a spin $\sig$ on $\SF'$ \st\ no two words as above cross each other. Note that this is an abstract property of sets of words, and it is defined without reference to the \Cg\ of $\PF$; in fact, it can be checked algorithmically. The essence of this paper is that this is enough to guarantee the planarity of the \Cg, and that a converse statement holds. 

This generalises an idea from \cite{vapf}, where it was shown that every planar discontinuous group admits a \splpr\ where every relator is \emph{facial}, i.e.\ it crosses no other word (where we consider words that are not necessarily among our relators).

Our actual definition of a \splpr, given in \Sr{secPlpr}, is in fact a bit more general than the above sketch. Consider for example the \Cg\ of the presentation $\left<a,b \mid a^n, b^2, aba^{-1}b\right>$. Its \Cg\ is a prism graph with an essentially unique embedding in $\R^2$. Note that the spin of half of its vertices is the reverse of the spin of other half. This is a general phenomenon: every 3-connected \Cg\ has an essentially unique embedding, and in that embedding all vertices have the same spin \emph{\utr}. However, for every generator $s$, either the two end-vertices of all edges labelled $s$ have the same spin, or they always have reverse spins. This yields a classification of generators into \emph{\spre} and \emph{\srev} ones, and our definition of a \splpr\ takes this into account; still, everything can be checked algorithmically.

The situation becomes much more complex however if one wants to consider planar \Cg s that are not 3-connected. Such graphs do not always have an embedding with all vertices having the same spin \utr; perhaps the simplest such example is the one of \fig{DSS}. 
\showFig{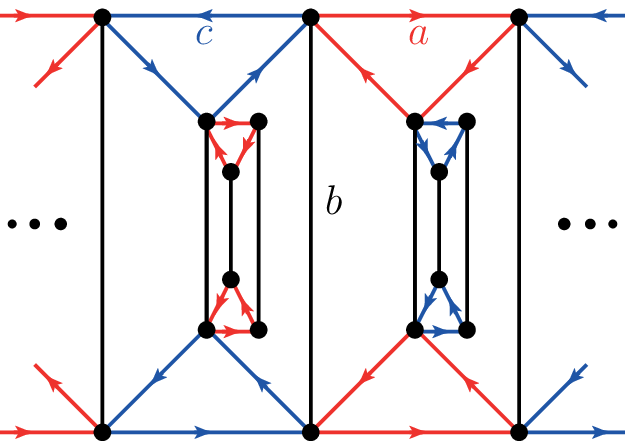}{A 2-connected planar \Cg\ from \cite{DrSeSeCon}, obtained by amalgamating two 6-element groups along an involution, which does not admit a consistent embedding.}

In order to capture such \Cg s we had to come up with the notion we call a \emph{\glplpr} (defined in \Sr{secConc}), which in particular translates, into  abstract, algorithmically checkable, properties of words as above situations as in \fig{DSS}, where a certain generator $s$ with $s^2=1$ separates the graph into two parts, and behaves in a \spre\ way in one part and in a \srev\ way in the other part. That such \glplpr s always give rise to planar \Cg s is the hardest result of this paper, many of its complications arising from the fact that given a \glplpr\ with such a `separating' generator $s$, it is impossible to predict whether $s=1$, which would imply that our \Cg\ does not quite have the structure anticipated by the presentation. The situation is complicated further by the fact that separating generators need not be involutions; an example is given in \fig{BlockEx}.

\showFig{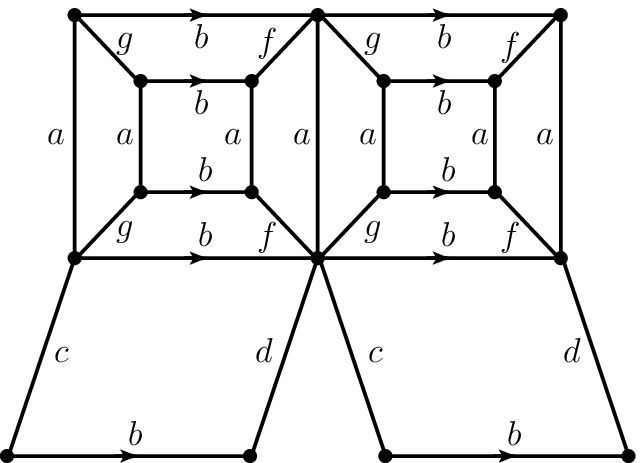}{An (infinite) planar \Cg, corresponding to the presentation $\left<a,b,c,d,f,g \mid a^2,c^2,d^2,f^2,g^2,(af)^2,(ag)^2,abab\inv gbfb\inv, cbdb\inv\right>$, with a separating edge $b$ which is not an involution.}

\medskip
This paper is structured as follows. After some general definitions in \Sr{sec_defs}, we introduce \gcplpr s in \Sr{secpltygen}, and show that every \Cg\ of every \gcplpr\ is planar in \Sr{secpltygenproof}. In \Sr{gplprconv} we prove the reverse direction, i.\,e.\ that every planar \Cg\ admits a \gcplpr. In \Sr{secConc} we slightly generalise from \emph{generic} to \emph{general} \plpr s, and 
put those facts together to obtain the results stated above. 
We finish with some open problems in Sections~\ref{secConc} and~\ref{secFR}.

\section{Definitions}\label{sec_defs}

\subsection{\Cg s and group presentations} \label{defpres}
We will follow the terminology of~\cite{diestelBook05} for graph-theoretical terms and that of~\cite{bogop} and~\cite{GroupsGraphsTrees} for group-theoretical ones. Let us recall the definitions most relevant for this paper.

A {\em group presentation} $\left< \SF \mid \mathcal \RF \right>$ consists of a set $\SF$ of distinct symbols, called the  {\em generators} and a set $\RF$ of words with letters in $\SF \cup \SF^{-1}$, where $\SF^{-1}$ is the set of symbols $\{s^{-1} \mid s \in \SF\}$, called  {\em relators}. Each such group presentation uniquely determines a group, namely the  quotient group $F_\SF/N$ of the (free) group $F_\SF$ of words with letters in $\SF \cup \SF^{-1}$ over the (normal) subgroup $N=N(\RF)$ generated by all conjugates of elements of $\RF$.

The \Cg\  $Cay(\PF)= Cay \left< \SF \mid \mathcal \RF \right>$  of  a group presentation $\PF= \left< \SF \mid \mathcal \RF \right>$
is an edge-coloured directed graph $\g= (V,E)$ constructed as follows. The vertex set of \g is the group $\Gam=F_\SF /N$ corresponding to $\PF$. The set of {\em colours} we will use is $\SF$.  For every $g\in \Gam, s\in \SF$ join $g$ to $gs$ by an edge coloured $s$ directed from $g$ to $gs$. Note that $\Gam$ acts on \g by multiplication on the left; more precisely, \fe\ $g\in \Gam$ the mapping from $V(G)$ to $V(G)$ defined by $x \mapsto gx$ is an {\em automorphism} of \G, that is, an automorphism of \g that preserves the colours and directions of the edges. In fact, \Gam\ is precisely the group of such automorphisms of \G. Any presentation of \Gam\ in which $\SF$ is the set of generators will also be called a presentation of $Cay(\PF)$.

Note that some elements of $\SF$ may represent the identity element of \Gam, and distinct elements  of $\SF$ may represent the same element of \Gam; therefore, $Cay(\PF)$ may contain loops and parallel edges of the same colour.

If $s\in \SF$ is an {\em involution}, i.e.\ $s^2=1$, then every vertex of \g is incident with a pair of parallel edges coloured $s$ (one in each direction). 
If $s^2$ is a relator in~$\RF$, we will follow the convention of replacing this pair of parallel edges by a single, undirected edge.
This convention is common in the literature \cite{LyndonSchupp}, and is convenient when  studying planar \Cg s.
\medskip

If \g is a \Cg, then we use $\Gam(G)$ to denote its group.

\medskip
If $R$ is any  (finite or infinite) word with letters in $\SF \cup \SF^{-1}$, and $g$ is a vertex of $G=Cay \left< \SF \mid \mathcal \RF \right>$, then starting from $g$ and following the edges corresponding to the letters in $R$ in order we obtain a walk $W$ in $G$. We then say that $W$ is {\em induced} by $R$ at $g$, and we will sometimes denote $W$ by $gR$; note that for a given $R$ \ta\ several walks in \g induced by $R$, one for each starting vertex $g\in V(G)$. 

Let $H_1(G)$ denote the first simplicial homology group of \g over $\Z$. 
We will use the following well-known fact which is easy to prove.
\begin{lem} \label{relcc}
Let $G= Cay \left< \SF \mid \RF \right>$ be a \Cg. Then the (closed) walks in \g induced by relators in $\RF$ generate $H_1(G)$. 
\end{lem}

\subsection{Graph-theoretical concepts}
Let $G=(V,E)$ be a connected graph fixed throughout this section. Two paths in \g are {\em independent}, if they do not meet  at any vertex except perhaps at common endpoints. 
If $P$ is a path or cycle we will use $|P|$ to denote the number of vertices in $P$ and  $||P||$ to denote the number of edges of $P$. Let $xPy$ denote the subpath of $P$ between its vertices $x$ and $y$.


A {\em hinge} of \g is an edge $e=xy$ \st\ the removal of the pair of vertices $x,y$  disconnects \G. A hinge should not be confused with a {\em bridge}, which is an edge whose removal separates \g  although its endvertices are not removed.

The set of neighbours of a vertex $x$ is denoted by {\em $N(x)$}.

\G\ is called {\em $k$-connected} if $G - X$ is connected for every set $X\subseteq V$ with $|X | < k$. Note that if \g is $k$-connected then it is also $(k-1)$-connected. The {\em connectivity $\kappa(G)$} of \g is the greatest integer $k$ such that \G\ is $k$-connected.

A $1$-way infinite path is called a {\em ray}. Two rays are {equivalent} if no finite set of vertices separates them. The corresponding equivalence
classes of rays are the {\em ends} of $G$. A graph is {\em multi-ended} if it has more than one end. Note that given any two finitely generated presentations of the same group, the corresponding \Cg s have the same number of ends. Thus this number, which is known to be one of $0,1,2, \infty$, is an invariant of finitely generated groups.

A {\em \dray} is a directed $2$-way infinite path.

The set of all finite sums of (finite) cycles forms a vector space over $\mathbb F_2$, the \emph{cycle space} of~$G$.

\subsection{Embeddings in the plane} \label{secDem}	

An {\em embedding} of a graph \g will always mean a topological embedding of the corresponding 1-complex in the euclidean plane $\R^2$; in simpler words, an embedding is a drawing in the plane with no two edges crossing.

A {\em face} of an embedding $\sig: G \to \R^2$ is a component of $\R^2 \sm \sig(G)$. The {\em boundary} of a face $F$ is the set of vertices and edges of \g that are mapped by \sig\ to the closure of $F$. The {\em size} of $F$ is the number of edges in its boundary. Note that if $F$ has finite size then its boundary is a cycle of \G.

A walk in \g is called {\em facial} with respect to \sig\ if it is contained in the boundary of some face of \sig. 

An embedding of a \Cg\ is called {\em \cns} if, intuitively, it embeds every vertex in a similar way in the sense that the group action carries faces to faces. Let us make this more precise.
Given an embedding \sig\ of a \Cg\ $G$ with generating set $S$, we consider \fe\ vertex $x$ of \g the embedding of the edges incident with $x$, and define the {\em spin} of $x$ to be the cyclic order of the set $L:=\{xy^{-1} \mid y\in N(x)\}$ in which $xy_1^{-1}$ is a successor of $xy_2^{-1}$ whenever the edge $xy_2$ comes immediately after the edge $xy_1$ as we move clockwise around~$x$. Note that the set $L$ is the same \fe\ vertex of $G$, and depends only on~$S$ and on our convention on whether to draw one or two edges per vertex for involutions. This allows us to compare spins of different vertices. Call an edge of \g {\em spin-preserving} if its two endvertices have the same spin in \sig, and call it {\em spin-reversing} if the spin of one of its endvertices is the reverse of the spin of its other endvertex. Call a colour in $S$ {\em \cns} if all edges bearing that colour are  spin-preserving or all edges bearing that colour are spin-reversing in \sig. Finally, call the embedding \sig\ {\em \cns} if every colour is \cns\ in \sig. Note that if \sig\ is \cns, then there are only two types of spin in \sig, and they are the reverse of each other.



The following classical result was proved by Whitney \cite[Theorem~11]{whitney_congruent_1932} for finite graphs and by Imrich \cite{ImWhi} for infinite ones.

\begin{thm} \label{imrcb}
Let \g be a 3-connected graph embedded in the sphere. Then every automorphism of \g maps each facial path to a facial path.
\end{thm}

This implies in particular that if \sig\ is an embedding of the 3-connected \Cg\ \G, then the cyclic ordering of the colours of the edges around any vertex of \g is the same up to orientation. In other words, at most two spins are allowed in \sig. Moreover, if two vertices  $x,y$ of \g that are adjacent by an edge, bearing a colour $b$ say, have distinct spins, then any two vertices $x',y'$ adjacent by a $b$-edge also have distinct spins. We just proved
\begin{lem} \label{lprem}
 Let \g be a 3-connected planar \Cg. Then every embedding of \g is \cns.
\end{lem}

\Cg s of connectivity 2 do not always admit a \cnsem~\cite{DrSeSeCon}. However, in the cubic case they do; see \cite{cay2con}.

An embedding is {\em Vertex-Accumulation-Point-free}, or {\em \vapf} for short, if the images of the vertices have no accumulation point in $\R^2$.

A {\em crossing} of a path $X$ by a path or walk $Y$ in a plane graph is a subwalk $Q=e\kreis{Q}f$ of $Y$ where the end-edges $e,f$ of $Q$ are incident with $X$ on opposite sides of $X$ (but not contained in $X$) and (the image of) $Q$ is contained in $X$ (\fig{XYQ}). Note that if $Q$ is a crossing of $X$ by $Y$, then $X$ contains a crossing $Q'=g\kreis{Q}h$ of $Y$ by $X$, which we will call the \emph{dual crossing} of $Q$.

\showFig{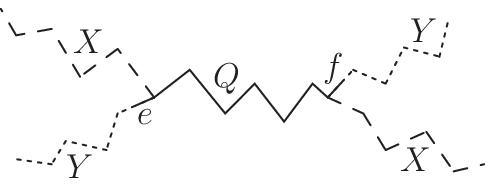}{A crossing of~$X$ by~$Y$.}

For a closed walk $W$ and $n\in\nat$, let $W^n$ be the $n$-times concatenation of~$W$.
Two closed walks $R$ and $W$ \emph{cross} if there are $i,j\in\nat$ such that $R^i$ contains a crossing of a subwalk of~$W^j$.
They are \emph{nested} if they do not cross.

\subsection{Fundamental groups of planar graphs}

Let $G$ be a graph.
The \emph{sum} of two walks $W_1,W_2$ where $W_1$ ends at the starting vertex of~$W_2$ is their concatenation.
Let $W=x_1x_2\ldots x_n$ be a walk.
Its \emph{inverse} is $x_n\ldots x_1$.
If $x_{i-1}=x_{i+1}$ for some~$i$, we call the walk $W':=x_1\ldots x_{i-1}x_{i+2}\ldots x_n$ a \emph{reduction} of~$W$.
Conversely, we \emph{add} the \emph{spike} $x_{i-1}x_ix_{i+1}$ to~$W'$ to obtain~$W$.
If $W$ is a closed walk, we call $x_i\ldots x_nx_1\ldots x_{i-1}$ a \emph{rotation} of~$W$.

Let $\VF$ be a set of closed walks.
The smallest set $\overline{\VF}\supseteq \VF$ of closed walks that is invariant under taking sums, reductions and rotations and under adding spikes is the set of closed walks \emph{generated by $\VF$}.
We also say that any $V\in\overline{\VF}$ is \emph{generated by~$\VF$}.
A closed walk is \emph{indecomposable} if it is not generated by closed walks of strictly smaller length.
Note that no indecomposable closed walk $W$ has a \emph{shortcut}, i.\,e.\ a (possibly trivial) path between any two of its vertices that has smaller length than any subwalk of any rotation of~$W$ between them.
In particular, indecomposable closed walks induce cycles.

For any $\eta\in\pi_1(G)$, let $W_\eta\in\eta$ be the unique reduced closed walk in~$\eta$ and $W_\eta^\circ$ be its (unique) cyclical reduction.
For $\VF\sub\pi_1(G)$, set
\[
\VF^\circ:=\{W_\eta^\circ\mid\eta\in\VF\}.
\]
By $\WF(G)$ we denote the set of all closed walks in~$G$.

The following theorem is an immediate consequence of~\cite[Theorem 6.2]{planarCycles2}, which is a generalisation of the main theorem of~\cite{planarCycles}.

\begin{thm}\label{thm_fundgr}
Let $G$ be a planar locally finite $3$-connected graph and $\Gamma$ a group acting on~$G$.
Then $\pi_1(G)$ has a generating set $\VF$ such that $\VF^\circ$ is a $\Gamma$-invariant nested generating set for $\WF(G)$ that consists of indecomposable closed walks.
\end{thm}

\section{\Plpr s} \label{secpltygen}

In this section we introduce our notion of planar presentation, which is the central definition of this paper. For the convenience of the reader, we start by recalling the definition of a special \plpr\ from \cite{planarpresI}. We then define the more involved \gcplpr s in \Sr{secPlprGen}.

\subsection{Special \plpr s} \label{secPlpr}

The intuition behind special \plpr s comes from the notion of a \cnsem\ given above: a \plpr\ is a group presentation endowed with some additional data (forming what we will call an \empr) which describe the local structure of a \cnsem\ of the corresponding \Cg, that is, the spin and the information of which generators preserve or reflect it. 

\medskip
Given a group presentation $\PF=\left<\SF \mid \RF\right>$, where $\SF$ is finite, or countably infinite,
we will distinguish between two types of generators $s$: those for which we have $s^2$ as a relator in $\RF$ and the rest. 
The reasons for this distinction will become clear later.
Generators $t$ for which the relation $t^2$ is provable but not explicitly part of the presentation might exist, but do not cause us any concerns. Given a group presentation $\PF=\left<\SF \mid \RF\right>$, we thus let $\IF=\IF(\PF)$ denote the set of elements $s\in \SF$ such that $\RF$ contains the relator $s^2$ or $s^{-2}$. 

Let $\SF'= \SF \cup (\SF \sm \IF)^{-1}$. For example, if $\PF=\left<a,b,c \mid a^2, b^2\right>$, then $\SF'= \{a,b,c,c^{-1}\}$.

A  {\em spin} on $\PF=\left<\SF \mid \RF\right>$ is a cyclic ordering of $\SF'$ (to be thought of as the cycling ordering of the edges that we expect to see around each vertex of our \Cg\ once we have proved that it is planar) 

An {\em \empr} is a triple $\PF, \sigma, \tau$ where  $\PF=\left<\SF \mid \RF\right>$ is a group presentation, $\sigma$ is a spin on $\PF$, and $\tau$ is a function from $\SF$ to $\{0,1\}$ (encoding the information of whether each generator is \spre\ or \srev).

To every \empr\ $\PF, \sigma, \tau$ we can associate a tree $\ttt$ with an \vapf\ embedding in $\R^2$.
As a graph, we let $\ttt$ be $Cay\left<\SF \mid s^2, s \in \IF\right>$. Easily, we can embed $\ttt$  in $\R^2$ in such a way that for every vertex $v$ of $\ttt$, one of the two cyclic orderings of the colours of the edges of $v$ inherited by the embedding coincides with $\sigma$ and moreover, for every two adjacent vertices $v,w$ of $\ttt$, the clockwise cyclic ordering of the colours  of the edges of $v$ coincides with that of $w$ if and only if $\tau(a)=0$ where $a$ is the colour of the $v$--$w$ edge. (If $\tau(a)=1$, then the clockwise ordering of $v$ coincides with the anti-clockwise ordering of $w$.)

Given a word $W$, we let $W^\infty$ be the 2-way infinite word obtained by concatenating infinitely many copies of $W$.
We say that two words $W,Z\in \RF$ {\em cross}, if there is a 2-way infinite path $R$ of $\ttt$ induced by  $W^\infty$ and a 2-way infinite path $L$ induced by $Z^\infty$ such that $L$ meets both components of $\R^2 \sm R$.
Note that, if two non-trivial words form closed walks in the \Cg, then the words cross if and only if the closed walks cross.

For example, consider the presentation $\PF=\left<n,e,s,w \mid n^2, e^2, s^2, w^2\right>$, the spin $n,e,s,w,n$ (read `north, east, south, west'), and $\tau$ identically 0. Then any word containing $ns$ as a subword crosses any word containing $ew$. The word $nesw$ however crosses no other word, and indeed adding that word to the above presentation yields a planar \Cg: the square grid.

\begin{defi} \label{defsplpr}
A  {\em special \plpr} is an \empr\ $(\PF, \sigma, \tau)$ such that 
\begin{enumerate}[(sP1)]
\item \label{plpri} no two relators $W,Z\in \RF$ cross, and 
\item \label{plprii} for every relator $R$, the number of occurrences of letters $s$ in $R$ with $\tau(s)=1$ (i.e.\ \srev\ letters) is even; here, the symbol $s^n$ counts as $|n|$ occurrences of $s$.
\end{enumerate}
\end{defi}

Requirement (sP\ref{plprii}) is necessary, as the spin of the starting vertex of a cycle must coincide with that of the last vertex. 

\medskip
In \cite{planarpresI} we proved the following results about \splpr s.

\begin{thm}[{\cite[Theorem~3.3]{planarpresI}}] \label{thm_3con}
Every planar, locally finite, $3$-connected \Cg\ admits a \splpr.
\end{thm}

\begin{thm}[{\cite[Theorem~4.2]{planarpresI}}] \label{thmplanar}
If $(\PF, \sigma, \tau)$ is a special \plpr, 
then its \Cg\ $Cay(\PF)$ is planar. Moreover, $Cay(\PF)$ admits a consistent embedding, with spin $\sigma$ and spin-behaviour of generators given by $\tau$.
\end{thm}


\subsection{General \plpr s} \label{secPlprGen}

We now extend the above definition of a \plpr, to a more general one, the advantage of which is that it can capture \Cg s with 2-separators that do not admit consistent embeddings, which will allow us to extend \Tr{thm_3con} and \Tr{thmplanar} to all planar \Cg s.

\medskip
Let again $\PF=\left<\SF \mid \RF\right>$ be a group presentation, and define $\SF'$ as above.

A {\em spin structure} $\CF$ on $\PF$ consists of a cover $B_1, \ldots, B_k$ of $\SF'$ (i.e.\ $\bigcup_i B_i = \SF'$) with the following properties
\begin{enumerate}[(S1)]
\item \label{ssiii} \fe\ generator $b$, the number of $B_i$'s containing $b$ equals  the number of $B_i$'s containing $b^{-1}$, and
\item \label{ssiv} the auxiliary graph $X$ on $\CF \cup \SF'$ with $s\sim B_i$ whenever $s\in B_i$, is a tree.
\end{enumerate}
(It will become clear later that a special \plpr\ is a special case of a general one when $\CF$ consists of a single set coinciding with $\SF'$.)

The {\em hinges} of this spin structure are the elements of $\SF'$ that have degree at least 2 in $X$; in other words, $h\in \SF'$ is a hinge if $h\in B_i\cap B_j$ for some $i\neq j$. Hinges of a spin structure correspond to edges of our \Cg\ \g whose two endvertices separate \G. 

For example, $a,b$ are the hinges of the presentation
\[
\left<a,b,c,d,f,g \mid a^2,c^2,d^2,f^2,g^2,(af)^2,(ag)^2,abab\inv gbfb\inv, cbdb\inv\right>
\]
given in \fig{BlockEx}, and $b$ is the only hinge in \fig{DSS}.  
The tree $X$ of condition~(S\ref{ssiv}) corresponding to the presentation of \fig{BlockEx} is shown in \fig{BaumEx2}. \fig{BaumEx} shows the corresponding tree $X$ that would result if we amalgamated the above group with two more groups each of which being isomorphic to the subgroup generated by $b,c,d$ along the subgroup spanned by~$b$. 

\showFig{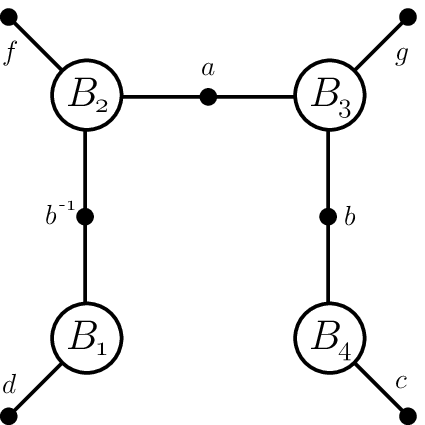}{An example: the tree $X$ of condition (S\ref{ssiv}) corresponding to  \fig{BlockEx}.}

\showFig{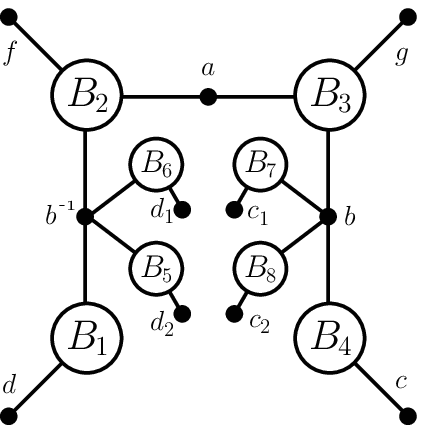}{The tree $X$ of condition (S\ref{ssiv}) corresponding to a variant of \fig{BlockEx}.}

Condition (S\ref{ssiv}) has the following important consequences:
\labtequ{ssi}{$B_i\cap B_j$ is either empty or a singleton \fe\ $i\neq j$,}
because if $h,g\in B_i\cap B_j$ then $h,g,B_i,B_j$ span a 4-cycle in $X$, which cannot happen when $X$ is a tree, and
\labtequ{ssii}{every $B_i$ contains at least one hinge unless $k=1$, i.e.\ $\CF$ is the singleton~${\{ \SF' \}}$,}
because if each neighbour of $B_i$ in $X$ has degree 1, then $B_i$ and its neighbours form a component of $X$.

\medskip
A \emph{generic \empr} is a quintuple $\PF, \CF, \sigma, \tau, \mu$ as follows; $\PF$~is a group presentation and $\CF$ a spin structure on $\PF$ as above; $\sigma$ is a function of $i\in \{1, \ldots, k\}$ assigning a spin (i.e.\ a cyclic ordering) to each $B_i\in \CF$;\\ $\tau \colon \SF \times \{1, \ldots, k\} \to \{0,1\}$ encodes the information of whether each generator is spin-preserving or spin-reversing in each $B_i$ it participates in (if $s\in \SF\sm B_i$, then the value of $\tau(s,i)$ will be irrelevant in the sequel); 
and \fe\ $b\in \SF$, and every $i$ for which $b\in B_i$, $\mu(b,i)$ is a $B_j$ such that $b^{-1}\in B_j$, and $\mu(b,i)\neq \mu(b,m)$ for $m\neq i$.
This $\mu$ encodes the information of which pairs of $B_i$ incident with the two endvertices of a given hinge belong to the same block of~\G. 
The use of $\SF$ rather than $\SF'$ in the definition of $\mu$ and $\tau$ is intended: the values we assign to each $b\in S$ give us enough information about how to treat~$b^{-1}$. 

For the time being, the data $\sigma, \tau, \mu$ are abstract objects describing the intended structure and embedding of our \Cg\ given by  $\PF$. But we will indeed prove that if these data satisfy certain conditions, then the \Cg\ is indeed planar and can be embedded in the intended way.

As an example, the presentation $\left<\SF \mid b^2, a^3, c^3, ab a^{-1}b, cbcb \right>$ of the graph  of \fig{DSS} can be endowed with the following data. The spin structure $\CF$ consists of two sets $B_1=\{b, c, c^{-1}\}, B_2= \{b, a, a^{-1}\}$. We can then let  $\sig(1)= (b, c, c^{-1})$, $\sig(2)=(b,a^{-1},a)$ ---but any other $\sig$ would do in this case as there are only two cyclic orderings of a set of three elements, and they are the reflection of each other--- $\tau(b,1)=0$, $\tau(b,2)=1$ ---this is the most interesting aspect of this graph: any $b$ edge is \spre\ in one of its incident blocks and \srev\ in the other--- and $\mu(b,1)=B_1$, $\mu(b,2)=B_2$ ---because $b$ stabilises the two components into which it splits the graph.

\medskip
Our general definition of a \plpr\ will be very similar to that of \Sr{secPlpr}, and still based on the idea of non-crossing relators. One difference is that we have to embed the tree $\ttt=Cay\left<\SF \mid s^2, s \in \IF\right>$ in~$\R^2$ more carefully: rather than demanding every vertex to have the same cyclic ordering of its incident colours in the embedding, which would in general make it impossible to adhere to the spin-behaviour encoded by $\tau$, we embed $\ttt$ (\vapf) in $\R^2$ in such a way that the following two conditions are satisfied.
Given a vertex $x\in V(\ttt)$ and $B_i\in \CF$, we write \emph{$B_i(x)$} for the edges of $x$ with labels in~$B_i$.
\begin{enumerate}[(B1)]
\item $\sig$ is respected, i.e.\ \fe\ vertex $x\in V(\ttt)$, and every $B_i\in \CF$, 
the cyclic ordering induced on $B_i(x)$ by our embedding coincides with $\sig(i)$ \utr. Moreover, the edges of $B_i(x)$ are consecutive in our embedding.
\item  $\tau$ is respected, i.e.\ \fe\ edge $e=vw$ of \ttt, and every $i$ \st\ the label $s$ of $e$ is in $B_i\in \CF$, we have $1_{\sig(i)} (B_i(v)) = 1_{\sig(j)} (B_j(w) )$ if and only if $ \tau(s,i)=0$,
where $B_j=\mu(s,i)$ and $1_{\sig(i)} (B_i(v))$ is 1 if the clockwise cyclic ordering of the colours of the edges of $B_i(v)$ coincides with $\sig(i)$ and 0 otherwise. 
\end{enumerate}

We repeat the definition of {\em crossing} from \Sr{secPlpr} verbatim:
given a word $W$, we let $W^\infty$ be the 2-way infinite word obtained by concatenating infinitely many copies of $W$.
We say that two words $W,Z\in \RF$ {\em cross}, if there is a 2-way infinite path $R$ of $\ttt$ induced by  $W^\infty$ and a 2-way infinite path $L$ induced by~$Z^\infty$ such that $L$ meets both components of $\R^2 \sm R$.


The second and final difference of our generalised definition of a \plpr\ compared to that of \Sr{secPlpr} will be an additional condition reflecting the idea that in a planar \Cg\ of connectivity 2, we can choose the relators in such a way that each cycle they induce is contained in a block. Recalling that our spin structure $\CF$ is intended to capture the decomposition into blocks, the following definition should not be too surprising. 

We say that a relator $R$ is {\em blocked} with respect to $\CF$, if it satisfies the following two properties. Firstly, for every two (possibly equal) consecutive letters $st$ appearing in $R^\infty$ or $(R^{-1})^\infty$, there is some $B_i\in \CF$ containing both $s^{-1}, t$.
Secondly,  \fe\ three consecutive letters $sbt$, where $b$ is a hinge, appearing in $R^\infty$ or $(R^{-1})^\infty$, if $B_i$ is the unique element of $\CF$ containing $s^{-1}, b$, then $\mu(b,i)$ contains both $b^{-1},t$, unless $s=b=t$ and $b^2\in \RF$; here, the existence of such a $B_i$ is guaranteed by the previous requirement, and its uniqueness is a consequence of~\eqref{ssi} in the definition of a spin structure.

\begin{defi} \label{defgplpr}
A  {\em  \gcplpr} is a generic \empr\ such that 
\begin{enumerate}[(P1)]
\item \label{gplprz} every relator in $\RF$ is blocked with respect to $\CF$;
\item \label{gplpri} no two relators $W,Z\in \RF$ cross; 
\item \label{gplprii} for every relator $R$, the number of occurrences of letters $t$ in $R$ with $\tau(t,i)=1$ (i.e.\ \srev\ letters), where $i$ is the unique value for which $s^{-1},t\in B_i$ for the letter $s$ preceding $t$ in $R$, is even\footnote{The existence and uniqueness of this $B_i$ is a consequence of~(P\ref{gplprz}); see the definition of `blocked'.}; here, the symbol $s^n$ counts as $|n|$ occurrences of $s$;
\item \label{gplpriii} no relator is a sub-word of a rotation of another relator.
\end{enumerate}
\end{defi}


Note that a \plpr\ as defined in \Sr{secPlpr} is a special case of a generic one when $\CF$ consists of a single set coinciding with $\SF'$.

\medskip
In \Sr{secConc} we will slightly generalise the notion of a \gcplpr\ further, by allowing the removal of certain redundancies, to obtain the notion of \glplpr\ discussed in the introduction.

\section{Proof of planarity of the \Cg\ of a \gcplpr} \label{secpltygenproof}

In this section we prove that the \Cg\ defined by any \gcplpr\ is planar (\Tr{Gplanar}). 

\bigskip
For a hinge $h\in \SF$, we let $\CF(h):= \{B_i \in \CF \mid h\in B_i \}$ and let $N(h)$ be the cardinality $|\CF(h)|$. Note that $|\CF(h)|= deg_X(h)$, where the tree $X$ is as in~(S\ref{ssiv}) of the definition of a spin structure.

Every hinge $b=xy\in E(\ttt)$ of \ttt\ labelled $h$ naturally splits \ttt\ into $N(h)$ subtrees: each of these subtrees contains $b$, it contains all edges of $x$ with labels in a component of $X - h$ containing some $B_i\in \CF(h)$ and no other edges of $x$, and it contains those edges of $y$ with labels in  the component of $X - h^{-1}$ containing $\mu(h, i)$ and no other edges of $y$; moreover, each such subtree is maximal with these properties. Let $Sep_b=\{T_1, T_2, \ldots, T_{N(h)}\}$ denote the set of those subtrees, and note that 
$\bigcap Sep_b = \{b\}$.

\begin{defi} \label{defprbl}
A \emph{\prbl} of \ttt\ is a maximal subtree $A \subseteq \ttt$ not separated by any $Sep_b$; that is, for every hinge $b$ of $\ttt$, $A$ is contained in some element of $Sep_b$.
\end{defi}

Alternatively, we can define a \prbl\ as a maximal subtree of \ttt\ such that for every $x,y\in V(A)$, if we let $s_1 s_2  \ldots  s_k$ denote the word (with letters in $\SF$) read along the $x$--$y$~path, then $s_{j-1}^{-1}, s_{j}$ lie in a common element of $\CF$ \fe\ $j>1$, and whenever $s_j$ is a hinge, and $s_{j-1}^{-1}, s_{j} \in B_i \in \CF$, then $s_{j}^{-1}, s_{j+1}\in \mu(s_j, i)$.

\comment{
Let $h\in \SF\sm \IF$ be a hinge. 
  is called \emph{\sese}, if 
Note that 
 $\mu$ induces a matching on $\{B_i \in \CF \mid h\in B_i \text{ or } h^{-1} \in B_i \}$, where we allow that matching to contain a loop...

We will group certain \prbl s of \ttt\ into larger structures called \emph{\prcl s}. These are defined in such a way that they can be embedded into $\R^2$ in a consistent way.

For every \nss\ hinge $h\in \SF$, we define a relation $\sim_h$ on the set of \prbl s of \ttt\ as follows. Given two \prbl s $A,B$ that intersect at a vertex $x$ (we allow $A,B$ to have a further vertex in their intersection), we declare $A \sim_h B$ if  $A$ contains $B_i(x)$ for some $i$ and $B$ contains $B_j(x)$ where $B_j=\mu(h,i)$. 
\begin{defi} \label{defprcl}
Let $\simeq_h$ denote the symmetric transitive closure of $\sim_h$. A \emph{\prcl} is the union of an equivalence class of $\simeq_h$.
\end{defi}
Note that any \prcl\ containing two incident hinges labelled $h,h^{-1}$ contains the whole $h$-labelled double-ray of \ttt\ containing those hinges, because the \prbl\ $B$ in the definition of $\sim_h$ contains $B_i(xh^{-1})$ by the definition of a \prbl, and so $B \sim_h C$ for the \prbl\ $C$ containing $B_j(xh^{-1})$. 

The following are immediate consequences of the definition of a \prcl:
\begin{enumerate}
 \item \label{cli} Every \prcl\ is a subtree of \ttt;
 \item \label{clii} no two distinct \prcl s share a \prbl;
 \item \label{cliii} \ttt\ is the union of its \prcl s; and
 \item \label{cliv} the label-preserving automorphisms of \ttt\ map \prcl s to \prcl s.
\end{enumerate}
}

\subsection{The embedding $\rho$ of \ttt} \label{secrho}

Recall that our proof of \Tr{thmplanar} starts with an embedding of the corresponding tree \ttt\ respecting the spin data. In our new setup of a \gempr\ our spin data give us some restrictions but do not uniquely determine an embedding of \ttt, and in fact we have to be careful with our choices in order for the proof in subsection~\ref{placlu} to work. 

Recall that our \gempr\ consists of the data $\PF, \CF, \sigma$, $\tau, \mu$. For $B\in \CF$ and a vertex $x\in V(\ttt)$, recall that $B_i(x)$ denotes the edges going out of $x$ whose labels are in $B$.
We claim that there is an embedding $\rho: \ttt \to \R^2$  satisfying all of the following (the first two were also used in the definition of crossing relators in \Sr{secPlprGen}).
\begin{enumerate}[($\rho$1)]
\item \label{rhosig} $\sig$ is respected, i.e.\ \fe\ vertex $x\in V(\ttt)$, and every $B_i\in \CF$, 
the cyclic ordering induced on $B_i(x)$ by $\rho$ coincides with $\sig(i)$ \utr. Moreover, the edges of $B_i(x)$ are consecutive in the spin of $x$ induced by~$\sig$.
\item \label{rhotau} $\tau$ is respected, i.e.\ \fe\ edge $e=vw$ of \ttt, and every $i$ \st\ the label $s$ of $e$ is in $B_i\in \CF$, we have $1_{\sig(i)} (B_i(v)) = 1_{\sig(j)} (B_j(w) )$ if and only if $ \tau(s,i)=0$,
where $B_j=\mu(s,i)$ and $1_{\sig(i)} (B_i(v))$ is 1 if the clockwise cyclic ordering of the colours of the edges of $B_i(v)$ coincides with $\sig(i)$ and 0 otherwise. 
\item \label{rhohin} $\mu$ is respected: let $b\in E(\ttt)$ be a hinge, and $U,W$ two paths containing $b$ contained in distinct \prbl s containing $b$. Then $U,W$ do not cross each other (at~$b$).
\item \label{rholocal} \label{rholast} If $x, y$ belong to the same $N(\RF)$-orbit (where $N(\RF)$ is the normal subgroup generated by $\RF$ as in \Sr{defpres}), and $b$ is a hinge at $x$ with label in $h\in\IF$, and $h\neq 1$, then the \emph{local spin} at $x$ \wrt\ $b$ coincides \utr\ with the local spin at $y$ \wrt\ the corresponding hinge labelled $h$.
\end{enumerate}
Here,  the \emph{local spin} with respect to a generator $h\in \SF'$ at a vertex  $x$ is the cyclic ordering on $N_{X}(h)$ 
induced by the embedding, where $X$ denotes the tree from \Sr{secPlprGen}.

If \g is a planar \Cg, then the results of \Sr{sec_Con2} imply that if we embed the universal cover \ttt\ of \g into $\R^2$ in a way that locally imitates an embedding of \G, then all above properties are satisfied.

An \emph{open star} is a subspace of a graph consisting of a single vertex and all open half-edges incident with it. A \emph{star} is the union of an open star with some of the midpoints in its closure.

Properties ($\rho$\ref{rhosig}) to ($\rho$\ref{rhohin}) are not hard to satisfy: we can embed $\ttt$ by starting with the star $E(o)$ and then recursively attaching the star $E(v)$ of a new vertex to the subtree embedded so far, and it is always possible to embed $E(v)$ without violating any of ($\rho$\ref{rhosig})--($\rho$\ref{rhohin}). In fact we could have several ways to extend the current embedding to $E(v)$, arising by `permuting' those $B_i(v), 1\leq i \leq k$ that do not contain the edge of $v$ embedded before, and  by `reflecting' any such $B_i(v)$. These choices are in direct analogy to the flexibility we have in the embedding of any planar \Cg\ of connectivity 2: permuting the  $B_i(v)$ corresponds to `activating' a hinge $b$ incident with $v$ to exchange the order in which blocks separated by $b$ are embedded. Reflecting a $B_i(v)$ corresponds to flipping such a block around.

These choices mean that ($\rho$\ref{rholocal}) will be violated unless we make them carefully. To achieve this, 
recall from~(S\ref{ssiv}) of \Sr{secPlprGen} that the auxiliary graph $X$ on $\CF \cup \SF'$ with $s\sim B_i$ whenever $s\in B_i$, is a tree. Let $\xl$ denote the tree obtained from~$X$ by attaching to each vertex $v$ in $\SF' \subset V(X)$ a new leaf, which leaf we denote by~$\ell(v)$. 

Fix an embedding $\chi: \xl \to \R^2$ of that tree with the following two properties. 
Firstly, the spin of any vertex $B\in \CF$ of \xl\ coincides with $\sig(B)$ \utr.

Recall that $N(v)=N_G(v)$ denotes the neighbourhood of $v$ in a graph $G$. For every hinge $h\in \SF \sm \IF$, note that $\mu(h,\cdot)$ defines a bijection between $N_X(h)$ and   $N_X(h^{-1})$ by the definition of $\mu$. We extend that bijection to $N_{\xl}(h)$ and $N_{\xl}(h^{-1})$ by mapping $\ell(h)$ to $\ell(h^{-1})$. The second property we impose on $\chi$ is that the spin it induces on $N_{\xl}(h)$ coincides \utr\ with the $\mu$-image of that spin induced by $\chi$ on $N_{\xl}(h^{-1})$, and this holds \fe\ such $h$.  

For an involution hinge $h\in \IF$, $\mu(h,\cdot)$ still defines a bijection between $N_X(h)$ and   $N_X(h^{-1})= N_X(h)$, and we do not impose any  requirement on $\chi$ as we did for $h\in \SF \sm \IF$. Instead, we let $\chi$ embed  $N_{\xl}(h)$ with an arbitrary spin $\phi= \phi(h)$, and define 
\begin{defi} \label{defphi}
The \emph{dual spin} of $\phi$ is the cyclic ordering on $N_{\xl}(h)$ obtained by composing $\phi$ with $\mu(h,\cdot)$.
\end{defi}
To satisfy ($\rho$\ref{rholocal}), we will construct $\rho$ in such a way that the local spin \wrt\ $h$ at every vertex in a given $N(\RF)$-orbit either always coincides with $\phi$ or it always coincides with the dual of $\phi$. We remark that we cannot construct $\rho$ algorithmically since we cannot predict which vertices of \ttt\ are in the same $N(\RF)$-orbit; we can only prove the existence of such a $\rho$ abstractly. 

 
\medskip
We think of this $\chi$ as providing instructions about how to construct $\rho$. As an example, if the set $\IF$ of involutions in $\SF$ is empty, then every vertex of \ttt\ will have the same spin \utr\ in $\rho$, and that spin can be read from $\chi$ by contracting all non-leaves of \xl\ into a single vertex; that vertex has the right spin in the resulting star.

Let $o=x_1, x_2, \ldots $ be an enumeration of $V(\ttt)$ \st\ $\{x_1, \ldots, x_k\}$ spans a connected subgraph for all $k$. We will construct $\rho$ by embedding the $x_i$ one at a time as indicated above. To begin with, we embed one edge $e_0$ incident with $x_1=o$ in the $0$th step. From now on, each step $i$ begins with some vertices being embedded fully, i.e.\ with all incident edges, and some vertices having exactly one of their edges  embedded in the current embedding $\rho_{i-1}$ of some subtree of \ttt. Let $j$ be the smallest index \st\ $x_j$ has exactly one of its edges $e_i$ embedded in $\rho_{i-1}$. We may assume without loss of generality that $j=i$ by changing our enumeration. 

We extend $\rho_{i-1}$ to $\rho_{i}$ by embedding the remaining edges incident with $x_i$. This will be done by the performing the following recursive procedure on \xl\ to obtain an embedded star $S_i$ with its edges labelled by $\SF'$, and then embedding $N_\ttt(x_i)$ with the same spin as $S_i$.

To begin with, let $\ell$ be the unique leaf of \xl\ such that $\ell= \ell(s)$ for the label $s\in \SF$ of the edge $e_i$ considered as outgoing from $x_i$. We distinguish the following cases.

{\bf Case 1}: If $s\not\in \IF$, and $s$ is a hinge, then we embed the star $N(s)$ of $s$ in \xl\ into $\R^2$ so that the spin of $s$ in this embedding coincides with the spin of $s$ in $\chi$ \utr; there are exactly two possibilities for this ---because of reflection--- and we choose the unique one guaranteeing ($\rho$\ref{rhohin}): unless we are in step $i=1$, in which case we just embed $N(s)$ with the spin of $s$ in $\chi$ without reflection, the other endvertex $x$ of $e_i$ has already been fully embedded, and the local spin \wrt\ $e_i$ (which now label $s^{-1}$ as seen from $x$) at $x$ coincides \utr\ with that induced on $N(s^{-1})$ by $\chi$ by induction hypothesis. We use the possibility to reflect or not in order to guarantee that the clockwise ordering of the $B_i$ in $N(s)$ coincides with the counterclockwise ordering of the $\mu(s,i)$ induced by the spin of $x$ in the embedding $\rho_{i-1}$.

{\bf Case 2}: If $s\not\in \IF$, and $s$ is not a hinge, then it has exactly two neighbours in $N(s)$ ($\ell(s)$ and the unique $B\in \CF$ containing $s$), and so reflection does not change the spin; we just embed $N(s)$ in the unique possible way.

{\bf Case 3}: If $s\in \IF$, and $s$ is not a hinge, then again we just embed $N(s)$ in the unique possible way.

{\bf Case 4}: Finally, if $s\in \IF$, and $s$ is a hinge, then we follow a similar approach to the $s\not\in \IF$ case, except that we now do not insist that the spin of $s$ in the embedding of $N(s)$ we produce coincides with the spin of $s$ in $\chi$ \utr; we just make sure that ($\rho$\ref{rhohin}) is satisfied, by embedding $N(s)$ so that the clockwise ordering of the $B_i$ in $N(s)$ coincides with the counterclockwise ordering of the $\mu(s,i)$ induced by the spin of $x$ in the embedding $\rho_{i-1}$; again this is well-defined unless we are in step $i=1$, in which case we just embed $N(s)$ with the spin of~$s$ in~$\chi$.

\medskip
Once $N(s)$ is embedded as above, we set $\xl_0:= N(s)$ and proceed by the following recursive procedure, which produces embeddings of an increasing sequence $\xl_1,\ldots, \xl_k(=\xl)$ of subtrees of \xl\ to embed the rest of \xl.
 
For $j=1,2,\ldots$, pick a leaf $v_{j}$ of $\xl_{j-1}$ which is not a leaf of \xl; if no such leaf exists then $\xl_{j-1}=\xl$ and we stop.
Then we extend the current embedding of $\xl_{j-1}$ by embedding $N(v_j)$ in such a way that the spin of $v_{j}$ coincides \utr\ with that induced by $\chi$, unless $v_{j}\in \IF\subseteq \SF$ and $v_j\neq 1$, in which case we do the following. 
Let $y_i=x_i v_j$ be the vertex of \ttt\ joined to $x_i$ by the edge labelled $v_j$. If no vertex of \ttt\ from the  $N(\RF)$-orbit of  $x_{i}$ or $y_i$ has been embedded yet by $\rho_i$, then we embed $N(v_{j})$ with local spin given by $\chi$. If some vertex of \ttt\ from the  $N(\RF)$-orbit of  $x_{i}$ has already been embedded by $\rho_i$, we embed $N(v_{j})$ with same spin \utr\ as we used so far for all $x_j, j<i,$ that are $N(\RF)$-equivalent to $x_{i}$; (we make this choice in order to satisfy ($\rho$\ref{rholocal})). Otherwise, we embed $N(v_{j})$ with the dual spin ---recall Definition~\ref{defphi}--- \utr\ of the spin we used so  far for all $x_j, j<i$ that are $N(\RF)$-equivalent to $y_{i}$. Note that these choices ensure that $N(v_{j})$ is embedded with the same spin \utr\ ---namely, either that induced by $\chi$ or its dual--- for all vertices in an $N(\RF)$-orbit, where we use the fact that, as $v_j\neq 1$, $x_i$ and $y_i$ are never in the same orbit.

In all cases, we still have the option of reflecting. If $v_j\in N(s)$, which means that $v_j\in \CF$ and $v_j$ contains the label $s$ of $e_i$, then we have to worry about satisfying ($\rho$\ref{rhotau}); but one of the two choices we have due to the option of reflecting will satisfy ($\rho$\ref{rhotau}) for $e=e_i$ and $B_i=v_j$ and we make that choice. (If $v_j\not\in N(s)$ then we do not worry about $\mu$ and $\tau$; the other endvertices of the edges incident with $x_i$ will make sure that this data is respected, just as we were careful above when embedding $N(s)$ for the label $s$ of $e_i$.)

Let $\xl_{j}:= \xl_{j-1} \cup N(v_{j})$.

\medskip

The procedure finishes when all of \xl\ has been embedded. Then, we contract all non-leafs of \xl\ to obtain the desired embedded star $S_i$ out of that embedding. Finally, we embed  $N_\ttt(x_i)$ with the same spin as $S_i$ to extend $\rho_{i-1}$ to~$\rho_{i}$. 

Let $\rho= \bigcup \rho_i$ be the limit of the $\rho_{i}$. We claim that $\rho$ satisfies conditions ($\rho$\ref{rhosig})--($\rho$\ref{rholast}). Indeed, if any of them is violated, then there is a first step in the above procedures violating it. But we designed all steps so that none of those conditions are violated: 
condition ($\rho$\ref{rhosig}) is never violated because we chose  $\chi$ so that the spin of every $B_i\in \CF$ coincides with $\sig(i)$ \utr, which implies that the corresponding edges of $x_i$ appear in that cyclic order \utr\ in $S_i$, and therefore in $\rho$, by the construction of the embedded star $S_i$.
Condition ($\rho$\ref{rhotau}) is never violated because of the way we embedded $N(v_{j})$ for  $v_j\in N(s)$ in the construction of $S_i$.
Condition ($\rho$\ref{rhohin}) is never violated because of the way we embedded $N(s)$ in the first step of the construction of $S_i$.
Finally, condition ($\rho$\ref{rholocal}) is never violated because of  the way we embedded $N(v_{j})$ for  $v_j\in \IF$ in the construction of $S_i$.

In fact, we obtain a slightly stronger property than ($\rho$\ref{rholocal}), and this will be useful later:
\labtequ{rholocal2}{Condition ($\rho$\ref{rholocal}) remains true if we define \emph{local spin} using $\xl$ instead of~$X$.}

\comment{	 $\tau: \SF \times \{1, \ldots, k\} \to \{0,1\}$ encodes the information of whether each generator is spin-preserving or spin-reversing in each $B_i$ it participates in (if $s\in \SF\sm B_i$, then the value of $\tau(s,i)$ will be irrelevant in the sequel); 
and \fe\ hinge $b\in \SF$, and every $i$ for which $b\in B_i$, $\mu(b,i)$ is a $B_j$ such that $b^{-1}\in B_j$, and $\mu(b,i)\neq \mu(b,k)$ for $k\neq i$.
This $\mu$ encodes the information of which pairs of $B_i$ incident with the two endvertices of a given hinge belong to the same block of \G. }

\subsection{Planarity of blocks} \label{placlu}

A \emph{block} of \g is an image $\pi([A])$ under the covering map $\pi$, where $A$ denotes a \prbl\ of \ttt\ and $[A]:=\{x\in V(\ttt) \mid x \simeq_N y \text{ for some $y\in A$}\}$ denotes its $N(\RF)$-equivalence class. 

Note that every block of \g is connected: given vertices $x,z$ in a block $K= \pi([A])$, we can find $x',z'\in A$ (and not just in the $N(\RF)$-orbit of $A$) with $\pi(x')=x, \pi(z')=z$, and so the $x'$--$z'$ path $P$ in $A$ yields the $x$--$z$ path $\pi(P)$ in~$K$.

\begin{lem} \label{clplanar}
Every block of \g is planar.
\end{lem}
%
In fact, we will prove a stronger statement similar to \Tr{thmplanar} (\cite[Theorem 4.2]{planarpresI}), namely, that every block admits an embedding into $\R^2$ respecting $\sig$ and $\tau$.

The proof of this follows the lines of our proof of the planarity of \g in the consistent case (\cite[Theorem 4.2]{planarpresI}), and we assume that the reader has already understood that proof. Here we will point out the differences.

Let $K$ be a \blk\ of \G. Let $D$ be a fundamental domain of $K$ in \ttt; that is, $D$ is a subset of $\ttt$ containing exactly one point from each $N(\RF)$-orbit $O$ such that $\pi(O)\in K$. With the same argument as in~\cite[Lemma 4.1]{planarpresI} we may assume that $D$ is connected since $K$ is. Moreover, we may assume without loss of generality that $D$ is a union of stars. Thus the closure $\cls{D}$ of $D$ in $\ttt$ is still the union of $D$ with all midpoints of edges that have exactly one half-edge in $D$, and $K$ can be obtained from $\cls{D}$ by identifying pairs of $N(\RF)$-equivalent midpoints.
As in the proof of \cite[Theorem 4.2]{planarpresI}, we will prove that any two pairs of such $N(\RF)$-equivalent midpoints are \emph{nested}, where we say that two pairs of midpoints $x,x'$ and $y,y'$ in $\overline{D}\setminus D$ are nested, if the $x$-$x'$ path in $D$ does not cross the $y$-$y'$ path, where we define crossing similarly to Section~\ref{secDem}.

In order to guarantee this nestedness, we will have to embed \ttt\ appropriately; 
in our general setup, \ttt\ cannot be embedded consistently as in the case of \splpr s, and this is why we are now only trying to prove the planarity of a \blk, and not of all of \g at once.

For a  relator $W$, we use $W_o$ to denote the closed walk $o_G W$ in \g induced by $W$ at $o_G$, and let $\ttt_W:=\pi^{-1}(W_o)$, which is a union of a set of double-rays of $\ttt$, which set we denote by $\ttt[W_o]$. 



Recall we have chosen an embedding $\rho$ of \ttt\ in \Sr{secrho}. For a \prbl\ $C$ of \ttt, we define a \emph{\suf} of $C$ to be a face of the embedding $\sig(C)$ of $C$ inherited by $\rho$. The \suf s of \ttt\ are the \suf s of all of the \prbl s of \ttt. Note that a \suf\ can contain several faces of~\ttt. 

The {\em dual} graph $\ttt^*$ of \ttt\ is the graph whose vertex set is the set of faces of \ttt, and two faces of \ttt\ are joined with an edge $e^*$ of~$\ttt^*$ whenever their boundaries share an edge $e$ of~$\ttt$.
For two faces $F,H$ of \ttt\ and an $F$--$H$~path $P_{FH}$ in $\ttt^*$, let $Cr(\ttt[W_o],P_{FH})$ denote the {\em number of crossings} of $\ttt[W_o]$ by $P_{FH}$; to make this more precise, for a double-ray $T$ in $\ttt[W_o]$, we write $cr(T, P_{FH})$ for the number of edges $e$ in $T$ such that $P_{FH}$ contains $e^*$, and we let  $Cr(\ttt[W_o],P_{FH}):= \sum_{T\in \ttt[W_o]} cr(T,P_{FH})$. We claim that
\labtequ{Cr}{for every two faces $F,H$ of \ttt, the parity of the number of crossings $Cr(\ttt[W_o],P_{FH})$ is independent of the choice of the path $P_{FH}$.}
To see this, note that if $C$ is a cycle in $\ttt^*$, then $Cr(\ttt[W_o],C)$ ---defined similarly to $Cr(\ttt[W_o],P_{FH})$--- is even because the embedding of \ttt\ is \vapf\ and so any ray entering the bounded side of $C$ has to exit it again. This immediately implies~\eqref{Cr}.

\comment{
REMOVE:\\
For reasons that will become clear later, it is convenient to introduce an alternative set of relators $\RF'$ of \G, defined as follows. This set of relators however is very similar to $\RF$, and the reader may assume that $\RF'=\RF$ and skip to the definition of $\sim$ below for the time being and come back later (namely, while reading the proof of \Lr{Ninv2}, which is the only occassion where this presentation is useful).

\begin{defi} \label{defR'}

For any $d,c\in S$ for which there is an $R\in \RF$ of the form $dUcU^{-1}=1$ ---where $U$ is an arbitrary, possibly trivial word--- we write $d \sim c$. Note that this is a symmetric relation. We let $\approx$ denote the transitive closure of $\sim$. Note that this is an equivalence relation on $\IF$. Then $\RF'$ is obtained from $\RF$ 
as follows. For every hinge $b\in \IF$ \st\ $b=1$, we add the relator $c=1$ for every $c$ in the $\approx$-equivalence class $[b]$, and we remove all relators of the form $dUcU^{-1}=1$ with $d,c\in [b]$ from  $\cR$.
\end{defi}

\begin{lem}\label{cPb}
The $Cay(\SF,\RF)= Cay(\SF,\RF')$.
\end{lem}
\begin{proof}
We need to prove that $c=1$ holds in the group of $\PF$ \fe\ $c\in [b]$, and that any word of the form $dUcU^{-1}$ with $d,c\in [b]$ can be proved to be equal to $1$ using relators of $\PF'$.

For the first statement, note that if $dUcU^{-1}\in \RF$, then $d=1$ implies $c=1$ and vice-versa. This means that either all elements of  $[b]$ are equal to 1 or none of them is by transitivity. But we are assuming that $b=1$, and therefore the former is the case.

This implies the second statement as well: if $bUcU^{-1} \in \RF$ and $d,c\in [b]$, then $b=c=1$ and so such relators are redundant when $b=1$ holds as they boil down to words of the form $UU^{-1}$.
\end{proof}

\begin{lem}\label{LRF'}
For $b\in \IF$ with $b=1$, and any relator $W$ in $\RF'$, the number of elements of $\ttt[W_o]$  containing any edge $e$ labelled by $b$ is even. 
\end{lem}
\begin{proof}
Since $b=1$, every element of $T\in \ttt[W_o]$ corresponds to another element $T'\in \ttt[W_o]$ obtained by translating $T$ from one endvertex of $b$ to the other. We have $T'=T$ if and only if $W$ is of the form $bUcU^{-1}$, because TODO. But as every such relator has been excluded from $\RF'$, the elements of $\ttt[W_o]$ containing $e$ come in pairs $T,T'$.
\end{proof}

$$==========================$$
We are now ready to define $\sim$. 
}
We will define our relation $\sim_K$, or just $\sim$ if $K$ is fixed, on the set of \suf s of \prcl s contained in  $\pii(K)$.
Given two \suf s $F,H$ lying in \prcl s contained in  $\pii(K)$, let $\ttt[W_o]_K$ denote the subset of $\ttt[W_o]$ 
contained in $\pii(K)$.
Now pick two faces $F'\subseteq F,H'\subseteq H$ contained in the \suf s $F,H$, and write $F\sim H$ if for each $F'$--$H'$ path $P_{F'H'}$ in $\ttt^*$, the number of crossings $Cr(\ttt[W_o]_K,P_{F'H'})$ of $\ttt[W_o]_K$ by $P_{F'H'}$ is even. Since $Cr(\ttt[W_o]_K,P_{F'H'})$ is independent of the choice of $P_{F'H'}$ by~\eqref{Cr}, it is also independent of the choice of $F',H'$, because if $F''$ is another face contained in~$F$, then the $F'$--$F''$ path of $\ttt^*$ contained inside $F$ crosses no element of $\ttt[W_o]_K$, because a \suf\ of any \prcl\ $C$ in $\pii(K)$ meets no element of $\ttt[W_o]_K$ by the definitions.


\subsubsection{The bipartitions $\{I,O\}$}

An important part of our planarity proof in the consistent case was that $\sim$ was invariant under the action of $N(\RF)$, see \cite[Lemma 4.4]{planarpresI}. Below (\Lr{Ninv2}) we prove an analogous statement for the general case, namely that the restriction of $\sim$ to the \suf s of the \prbl s in $\pii(K)$ is $N(\RF)$-invariant.

The rest of our proof is almost identical to that of \cite[Theorem 4.2]{planarpresI}, except that we are now working with the \blk\ $K$ of \g rather than the whole graph. 

The equivalence relation $\sim$, now restricted on the set of \suf s $\FF$ of $\pii(K)$, uniquely determines a bipartition $\{I,O\}$ on $\FF$ by choosing one \suf\ $F\in \FF$ and letting $I:=\{H\in \FF \mid H \sim F\}$ and $O:= \FF \sm I$.

\medskip

Next, we adapt the material of \cite[Section 4.3.1]{planarpresI} to our new setup. For every \suf\  $F$ in $\pii(K)$, glue a copy of the domain $\overline{F}\subset \R^2$ to $K$ by identifying each point of $\partial F$ with $\pi(\partial F)$.
If $F,F'$ are equivalent face boundaries, in other words, if $\pi(\partial F)=\pi(\partial F')$, then we identify the corresponding 2-cells glued onto~$K$. Let $\cellsk$ denote the set of these 2-cells, and let $\Kc=K \cup \cellsk$ denote the 2-complex consisting of $K$ and these 2-cells. 


\Lr{Ninv2} now means that if $Z$ is a closed walk of \g (here we really mean~$G$ and not just $K$) induced by a relator, then $\{{I},{O}\}$ induces a bipartition $\pi[I],\pi[O]$ of $\cellsk$. 
Let us still denote this bipartition of $\cellsk$ by {\em $B_Z$}.

We extend that bipartition to an arbitrary cycle in $K$: given a cycle $C$ of $K$, we choose a `proof' $P$ of $C$; that is, a sequence of closed walks  $W_i, 1\leq i \leq k$ of \g induced by rotations of relators such that $C=\sum_{1\leq i\leq k} W_i$. The existence of such a sequence $(W_i)$ is not affected by the fact that we are focusing on a subgraph $K$; the $W_i$ are allowed to be arbitrary relators. For every $W_i$, let $I_{W_i},O_{W_i}$ denote the two sides of the bipartition $B_{W_i}$ of $\cellsk$ from above, and define the bipartition $B_C:=\{I_C,O_C\}$ of  $\cellsk$  by $I_C:= \sydi_i I_{W_i}$ and $O_C:= \cells \sydi I_C$.

While in the definition of $B_C$ it appear that it depends on the proof $P$, it actually does not as we shall see later. Until then, we denote it by $B_C(P)$ to make it clear that it depends on~$P$.
Our next aim is to show that, in a certain way, $B_C(P)$ behaves like the bipartition of the faces of a plane graph induced by a cycle $C$: to move between the two sides, one has to cross an edge of $C$. This is achieved by \Lr{faces} below, for the proof of which we need the following. 

\begin{lem} \label{We}
Let $e$ be a directed edge of~$K$, let $W\in \RF$ be a relator which is not of the form $b^2=1$ for $b\in \SF$, and let $o_K W$ be the closed walk of $K$ rooted at some vertex $o_K$ of~$K$ induced by $W$. Then the number of double-rays in $\ttt[W_o]$ containing $e$ equals the number of times that $o_K W$ traverses $\pi(e)$. 
\end{lem}
\begin{proof}
If $o_K W$ does not traverse $\pi(e)$ then $\ttt[W_o]$ avoids $e$ and we are done. So suppose that $o_K W$ does traverse $\pi(e)$.
Let $o_K W^\infty$ denote the two-way infinite walk on~$K$ obtained by repeating $o_K W$ indefinitely. Let $T\in \ttt[W_o]$ be the lift of $o_K W^\infty$ to \ttt\ (via $\pi^{-1}$) sending $\pi(e)$ to $e$, and note that $T$ is a double-ray containing $e$. Let $Q$ be the subpath of $T$ that starts with $e$ and finishes when a rotation of the word $W$ is completed. By the definition of $\ttt[W_o]$, there is a 1--1 correspondence between the elements of $\ttt[W_o]$ containing $e$ and the directed edges $e'$ in $Q$ that are $N(\RF)$-equivalent to $e$: each such element of $\ttt[W_o]$ can be obtained by translating $T$ by the automorphism of \ttt\ sending $e'$ to $e$. 

Now note that $o_K W$ traverses $\pi(e)$ whenever its lift $T$ traverses one of those~$e'$. Combined with the above observations this proves our assertion.
\end{proof}

\begin{lem} \label{faces}
For every $e\in E(K)$,  the bipartition  $B_C(P)$ separates 2-cells of $e$ if and only if $e\in C$.
\end{lem}
\begin{proof}
Let $I,O$ be the two elements of $B_C(P)$ as defined above. Then, letting $1_{F\in I}$ denote the indicator function of $F\in I$, we have  
$$1_{F\in I} = N_F:= | \{W_i \mid F\in I_{W_i}\} | \pmod 2,$$ 
and similarly 
$$1_{H\in I} = N_H:= | \{ W_i \mid H\in I_{W_i}\} | \pmod 2.$$
But 
$$N_F+N_H = | \{ W_i \mid W_i \text{ separates } F \text{ from } H\} | \pmod 2 $$ by the construction of $I,O$. We claim that $| \{ W_i \mid W_i \text{ separates } F \text{ from } H\} |$ is odd if and only if $e\in E(C)$. Indeed, $B_{W_i}$ separates $F$ from $H$ exactly when $W_i$ traverses $e$ an odd number of times by
\labtequ{FeH}{for every edge $e$ of \ttt, the two faces $F,H$ of $e$ lie in distinct elements of $\{{I},{O}\}$ if and only if $e\in \ttt_W$ and $e$ lies in an odd number of elements of $\ttt[W_o]$}
and \Lr{We}, 
and $e$ is in $C$ exactly when there is an odd number of $W_i$ that traverse $e$ an odd number of times. 

Since that number is even if $e\not\in E(C)$ and odd otherwise, our last congruence yields $N_F+N_H = 1 \pmod 2$ if and only if $e\in E(C)$. Therefore, the previous congruences imply that $1_{F\in I}= 1_{H\in I}$ if  $e\not\in E(C)$ and $1_{F\in I}\neq 1_{H\in I}$ if  $e\in E(C)$, which is our claim.
\end{proof}

\Lr{faces} implies in particular that $B_C(P)$ is characterised by $C$ alone and is therefore independent of $P$, since $\Kc$ was defined without reference to $P$. Thus we can denote it by just $B_C$ from now on.

In the following, we use again the definition of a crossing from \Sr{secDem}.

\begin{lem} \label{crossep}
Let $C'$ be a finite path  of \ttt\ such that $C:=\pi(C')$ is a cycle of~$K$, and let $Q=e\kreis{Q}f$ be a crossing of $C'$ in \ttt. Then $B_C$ separates the 2-cells incident with $\pi(e)$ from the 2-cells incident with $\pi(f)$. Moreover, if $Q_2$ is a path of \ttt\ such that $\pi(Q_2)$ is a cycle of~$K$, then $Q_2$ crosses $C'$ an even number of times.
\end{lem}
\begin{proof}
Let $F$ be a face incident with the first edge $e$ of $Q$, and let $H$ be a face incident with the last edge $f$ of $Q$.
By the definition of a crossing, we can find a finite sequence $(F=)F_1, \ldots, F_k(=H)$ of faces of \ttt\ such that each $F_i$ shares an edge $e_i$ with $F_{i+1}$ and exactly one of the $e_i$ lies in $C'$: we can visit all faces incident with $Q$ until we reach~$H$. By \Lr{faces} and \Lr{Ninv2}, $B_C$ separates $\pi(F_1)$  from $\pi(F_k)$. 
This proves our first assertion.
\medskip

For the second assertion, note that $\pi(Q_2)$ can be written as a concatenation of subarcs $C_1 D_1 C_2 D_2 \ldots C_k=C_1$ where each $C_i$ lifts to a crossing of $C'$ by $Q_2$ and each $D_i$ avoids $C$ and shares exactly one end-edge with each of $C_{i}$ and $C_{i+1}$. 
We proved above that the 2-cells incident with end-edges of each $C_i$ are separated by $B_C$. The same arguments imply that  the 2-cells incident with end-edges of each $D_i$ are {\em not} separated by $B_C$. Since $\pi(Q_2)$ is a cycle, this implies that $Q_2$ crosses $C'$ an even number of times. 
\end{proof}

As in the end of the proof of \Tr{thmplanar}, the last lemma says that any two cycles of $K$ cross each other an even number of times, and therefore any two pairs of identified points of $\cls{D}$ are nested.

This completes the proof of \Lr{clplanar}, except that we still have to prove the two lemmas we used above:


\begin{lem}\label{evenTb}
For $b\in \IF$ with $b=1$, and any relator $W$ in $\RF$, the number of elements of $\ttt[W_o]$  containing any edge $e$ labelled by $b$ is even. 
\end{lem}
\begin{proof}
Let $T$ be an element of $\ttt[W_o]$ containing $e$.
The automorphism $\beta$ of \ttt\ exchanging the two endvertices of $e$ maps  $T$ to an element $T'$ of $\ttt[W_o]$ because $b=1$ and so the two end-vertices of $e$ are $N(\RF)$-equivalent. Note that $T\neq T'$ even if $T,T'$ contain the same vertices, because they have opposite directions (remember that double-rays are directed by definition). Note that $\beta(T')=T$. Therefore, $\beta$ establishes a bijection without fixed points on the elements of $\ttt[W_o]$ containing $e$, which means that the number of those elements is even.
\end{proof}

\begin{lem} \label{Ninv2}
\Fe\ \blk\ $K$ of \G, the restriction of $\sim_K$ to the \suf s of $\pii(K)$ is invariant under the action of $N(\RF)$ on~\ttt. 
\end{lem}
\begin{proof}
We will adapt the proof of \cite[Lemma 4.4]{planarpresI}. Since $K$ is fixed, let us just write $\sim$ instead of $\sim_K$.

We need to prove that if $F,H$ are \suf s of $\pii(K)$ in the same orbit of $N(\RF)$, then $F\sim H$. Again, we may assume that there are vertices $x,y$ in the boundaries of $F,H$ respectively, such that $y= x w R w^{-1}$ for some word $w$ and some relator $R\in \RF$: by the definition of the normal closure $N(\RF)$, if we can prove $F\sim H$ in this case, we can prove $F\sim H$ for every two  $F,H$ in the same orbit of $N(\RF)$.

Let $\alpha_{FH}$ be the automorphism of \ttt\ mapping $x$ to~$y$. 

Decompose the path $Q:=x w R w^{-1}$ into (inclusion-)maximal subpaths contained in a \prbl. Then we can write$$Q= P_1 \cup P_2 \cup \ldots \cup P_k(=P'_k) \cup P'_{k-1} \cup \ldots \cup P'_1,$$ where the $P_i, P'_i$ are those maximal subpaths, $P'_i$ is $N(\RF)$-equivalent to $P_i$ \fe\ $i<k$, and $P_k$ contains the subpath of $Q$ induced by $R$ (such a $P_k$ exists because every relator $R$ is blocked). Note that the intersection of any two subsequent $P_i$ or $P'_i$ is either a hinge separating the corresponding \prbl s, or a single vertex incident with such a hinge.

Since we are free to choose any $F$--$H$ walk $P_{FH}$ in $\ttt^*$ to decide whether $F\sim H$, we will choose a convenient one, which we construct now.

Recall that every $P_i, i>1$ starts and ends at hinges, which we will call $h_{i-1}, h_i$, separating its \prbl\ from the \prbl s containing $P_{i-1},P_{i+1}$ respectively; here $h_{i-1}, h_i$ may or may not be contained in $P_i$ as end-edges. 

Let $C_i$ be the \prbl\ containing $P_i$ and let $C'_i$ be the \prbl\ containing~$P'_i$.

Let $\Pi_i, k>i>1$, be an (inclusion-)minimal path in $\ttt^*$ joining a \suf\  incident with $h_{i-1}$  to a \suf\ incident with $h_{i}$ ---where we say that a \suf\ $F$ is \emph{incident} with an edge if the boundary of $F$ contains that edge---   
\st\ all vertices of $\Pi_i$ are faces sharing a vertex with $P_i$, and $\Pi_i$ does not intersect $P_i$ (at a midpoint of any edge); see \fig{PFH}. Define $\Pi'_i$ similarly using $P'_i$ instead of $P_i$. Note that there are exactly two such paths $\Pi_i$ to choose from, one on either side of~$P_i$; it doesn't matter much which of the two we will choose, but let us make `the same' choice for both $\Pi_i$ and $\Pi'_i$; more precisely, we ensure that
\labtequ{zeta}{$\Pi_i$ crosses an edge $e$ of $C_i$ (incident with $P_i$) \iff\ $\Pi'_i$ crosses the edge $\alpha_{FH}(e)$ of $C'_i$.}
%
%
This is possible because $\rho$ embeds $C_i$ the same way as $C'_i$ \utr, and $\Pi_i$ is uniquely determined once we choose which of the two \suf s of $C_i$ incident with $h_i$ we want it to contain; by choosing $\Pi'_i$ to contain the corresponding \suf\ incident with $h'_i$, our claim is satisfied. 
Note that $\Pi_i$ does not cross $h_i$, because if it did we could shorten it. 

For $i=1$ we let $\Pi_1$ be a minimal path in $\ttt^*$ joining $F$ to a \suf\ incident with $h_{1}$, 
and otherwise be defined similarly to $\Pi_i, k>i>1$. Define $\Pi'_1$ similarly. 
Finally, let $\Pi_k=\Pi'_k$ be a minimal path in $\ttt^*$ joining a \suf\ incident with $h_{k-1}$ to a \suf\ incident with $\alpha_{FH}(h_{k-1})$ without crossing~$P_k$. 


Let $\sqcup_i, k> i\geq 1$ be a path in $\ttt^*$ joining the last vertex of $\Pi_i$ to the first vertex of $\Pi_{i+1}$ \st\ all vertices of $\sqcup_i$ are faces sharing a vertex with $P_i \cap P_{i+1}$, and define $\sqcup'_i$ similarly for $\Pi'_i$, $\Pi'_{i+1}$; there are several choices for this $\sqcup_i$, so let us make it uniquely determined: if $P_i \cap P_{i+1}$ is a single vertex, then there are two candidates, and we always choose the one crossing $h_i$.
If $P_i \cap P_{i+1}$ is the hinge $h_i$, then there are up to four choices, and we choose the one that crosses $h_i$ and is contained in the two \suf s of $C_i$ incident with  $h_i$ and in the two \suf s of $C_{i+1}$ incident with  $h_i$.

It follows from the choice of $\sqcup_i$ that it behaves well \wrt\ elements of~$\CF$:
\labtequ{beta}{If $\sqcup_i$ meets an edge  in $B_i(v)\sm \{h_i\}$ (where $B_i\in \CF$) where the vertex $v$ is incident with $h_i$, then  $\sqcup_i$ meets every edge of $B_i(v)$.}

A similar but slightly stronger is true for $\Pi_i$:
\labtequ{gamma}{If $\Pi_i$ meets an edge lying inside some \suf\ of $C_i$, then $\Pi_i$ visits all faces incident with $P_i$ inside that \suf.
}
Indeed, 
$\Pi_i$ is by definition a minimal path joining certain \suf s of $C_i$; therefore, it crosses any \suf\ either completely or at a single boundary edge.


Finally, we obtain $P_{FH}$ by concatenating all the $\Pi_i,\sqcup_i,\Pi'_i $ and $\sqcup'_i$: 
$$P_{FH}:= \Pi_1 \cup \sqcup_1 \cup \Pi_2 \ldots \cup \sqcup_{k-1} \cup \Pi_k(=\Pi'_k) \cup \sqcup'_{k-1} \ldots \cup \sqcup'_1 \cup \Pi'_1.$$

   \begin{figure}[htbp]
   \centering
   \noindent
\includegraphics[width=.95\textwidth]{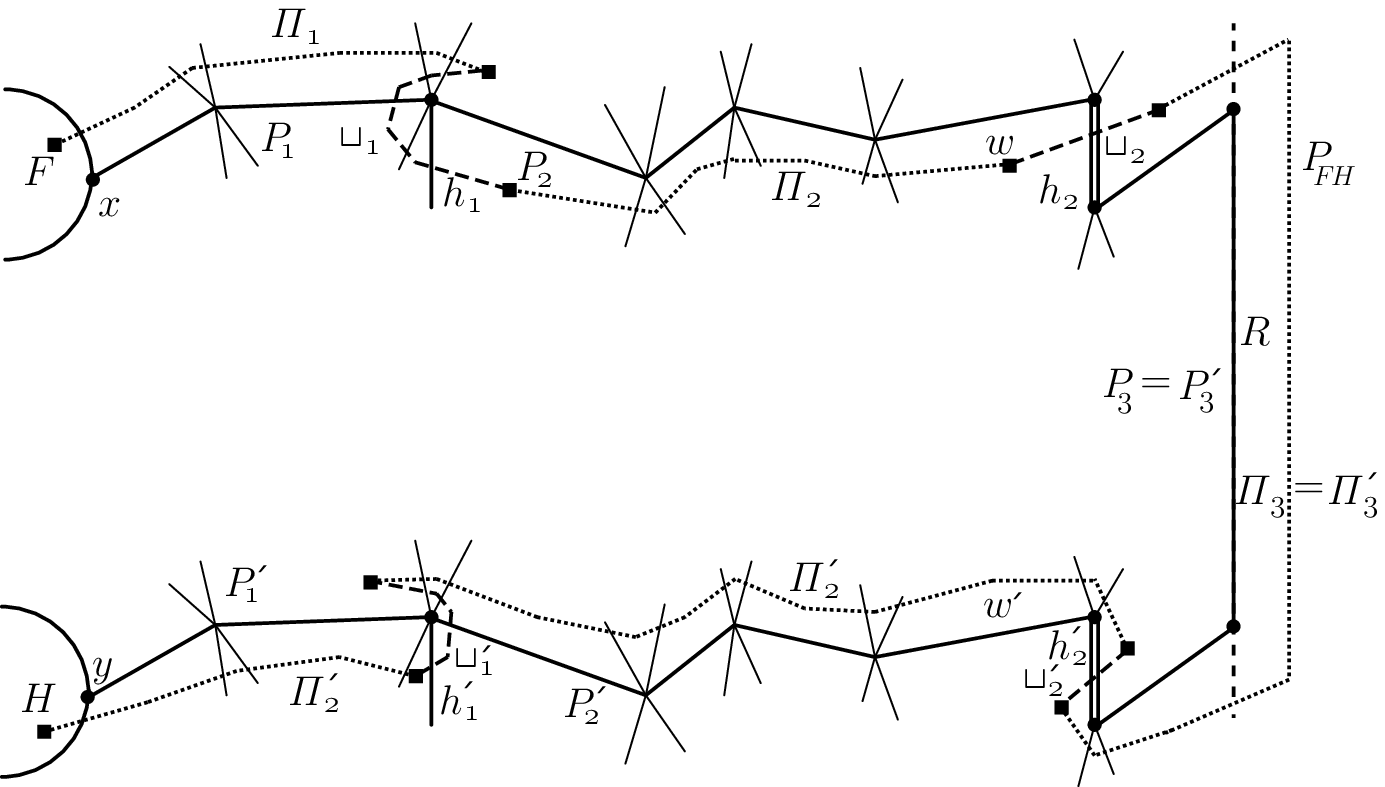}
   \caption{\small The path $P_{FH}$ (dashed) in the proof of \Lr{Ninv2} with the paths $\Pi_i$, $\Pi_i'$, $\sqcup_i$, $\sqcup_i'$.}
   \label{PFH}
   \end{figure}

We need to check that $Cr(\ttt[W_o]_K,P_{FH})$ is even. We will do so by showing that the contributions of the $\Pi_i$ to $Cr(\ttt[W_o]_K,P_{FH})$ cancel with those of the $ \Pi'_i$, 
and the contributions of the $\sqcup_i$ cancel with those of the $\sqcup'_i$.

Let $T$ be an element of $\ttt[W_o]_K$ with odd $cr(T,P_{FH})$, i.\,e.\ with an odd number of crossings of $T$ by $P_{FH}$; only such $T$ matter. 
Let $T':= \alpha_{FH}(T)$.

Let us first consider the total number of crossings of such $T$ by the subpaths $\Pi_i, \Pi'_i, i<k$, of $P_{FH}$.

If $T$ is contained in $C_i$, 
then $cr(T,\Pi_i)= cr(T',\Pi'_i)$ by \eqref{zeta}. 

{If $T$ is not contained in $C_i$, then $\Pi_i$ crosses $T$ an even number of times (0 or 2):}
this is easy to see when $T\cap P_i$ is a single vertex $v$ by applying \eqref{gamma} to that vertex. The situation is slightly subtler when $T\cap P_i$ is a hinge $g$ ---no other option is possible as distinct \prbl s intersect at an edge at most by construction. In this case, we remark that the \prbl\ $D$ containing $T$ lies in some \suf\ of $C_i$ by the construction of $\rho$, and again $\Pi_i$ must cross all faces incident with~$g$ inside that \suf\ by \eqref{gamma}, therefore crossing both edges of $T$ incident with~$g$.

Finally, it is not hard to see that $\Pi_k=\Pi'_k$ has an even contribution to $Cr(\ttt[W_o]_K,P_{FH})$.

These facts combined show that $\sum_{T\in \ttt[W_o]_K} cr(T, \bigcup_i \Pi_i)$ is even.

\medskip
Next, we consider the total number of crossings of such $T$ by the subpaths $\sqcup_i,\sqcup'_i$.
Suppose $cr(T,\sqcup_i)$ is odd. Then it must equal 1 as $\sqcup_i$ is too short to cross a double-ray three times, where we used  property ($\rho$\ref{rhohin}) of  our embedding $\rho$ that \prbl s do not cross each other. 

Let $v_i$ be the last vertex of $P_i$ and $v'_i$ the last vertex of $P'_i$.
If the local spin at $v_i$ \wrt\ $h_i$ coincides \utr\ with  the local spin at $v'_i$ \wrt\ $h'_i$, then $cr(T,\sqcup_i)= cr(T',\sqcup'_i)$ (here, local spin refers to \xl\ rather than $X$; recall \eqref{rholocal2}). Therefore, the total contribution of the pair $T,T'$ to 
$Cr(\ttt[W_o]_K,P_{FH})$ is even and can be ignored.

If those local spins do not coincide \utr, then by the choice of~$\rho$ ($\rho$\ref{rholocal}), the label of $h_i$ is an involution $b\in\IF$ with $b=1$. In this case however, \Lr{evenTb} applies, yielding that the set $H$ of elements of $\ttt[W_o]_K$ containing $h_i$ is even. We claim that $T\in H$ (i.e.\ $h_i \subset T$): this follows from $cr(T,\sqcup_i)=1$, the fact that $\sqcup_i$ only contains faces of \ttt\ incident with $h_i$ by its construction, and \eqref{beta}. Moreover, \eqref{beta} also implies that $cr(R,\sqcup_i)=1$ \fe\ other $R\in H$. But as $|H|$ is even, the total contributions $\sum_{R\in H} cr(R,\sqcup_i)$ of its elements are even and can be ignored as well.

Summing up, we proved that both
\[
\sum_{T\in \ttt[W_o]_K} cr(T, \bigcup_i \Pi_i)\quad\text{ and }\quad\sum_{T\in \ttt[W_o]_K} cr(T, \bigcup_i \sqcup_i)
\]
are even. Therefore $Cr(\ttt[W_o]_K,P_{FH})$ is even as well, since it is the sum of those two sums by definition.
\end{proof}

\subsection{From the planarity of blocks to the planarity of \G}

The main aim of this section is to prove

\begin{lem}\label{hinsep}
Every hinge of \g separates its incident blocks. 
\end{lem}
\begin{proof}
The statement is equivalent to the statement that every cycle of \g \emph{crosses} each hinge $b$ an even number of times, where the number of crosses of $b$ by $C$ is the maximum number of edge disjoint subpaths $P_i$ of $C$ such that $b$ separates each $P_i$ into two (possibly trivial, but non-empty) subpaths that lie in distinct blocks.

To prove the latter, let $C= c_0 c_1 \ldots c_k$ with $c_k=c_0$ be a cycle, and let $L=t_0 t_1 \ldots t_k$ be a lift of $C$ to \ttt\ via $\pii$.
Fix a hinge $b$. We may assume without loss of generality that $c_0$ is not a vertex of $b$.
Let $P=w_1R_1 w^{-1}_1 \ldots w_k R_k w^{-1}_k$ be a proof of~$C$ in our presentation.

Since $c_0\not\in b$ and since the end vertices of $w_iR_iw_i\inv$ are $N(\RF)$-equivalent to~$c_0$, any crossings of $b$ by $P$ occur inside the subpaths $w_i R_i w^{-1}_i$ and not when switching from $w_{i-1} $ to $w_i$.  We have no crossings of $b$ inside any $R_i$ because our relators are blocked. Moreover, any crossings of $b$ inside a $w_i$ are paired up by crossings of $b$ inside $w^{-1}_i$.
Thus the number of crossings of $b$ by $P$, and hence by~$C$, is even.
\end{proof}

This, combined with the planarity of blocks we proved in the previous section, easily implies the planarity of~\G:

\begin{thm} \label{Gplanar}
Let \g be the \Cg\ of a \gcplpr. Then \g is planar.
\end{thm}
\begin{proof}
Combining \Lr{clplanar} with \Lr{hinsep} easily yields that \g is planar. Indeed, 
we can embed \g one block at a time: since incident blocks share a hinge only by \Lr{hinsep}, if we have already embedded a block $A$ meeting a block $B$ at a hinge $b$, then it is easy to embed $B$ inside one of the two faces (we are free to choose) of the current embedding whose boundary contains $b$.
\end{proof}

\section{Every planar \Cg\ admits a \gcplpr} \label{gplprconv}
In this section we prove the converse of \Tr{Gplanar}, namely that every planar \Cg\ admits a \gcplpr.

We start by showing that every planar \Cg\ of connectivity~1 can be extended into a 2-connected one using redundant generators; see \Lr{cg1con} below. 
We then show that every 2-connected planar \Cg\ admits a \gcplpr\ in \Sr{sec_Con2}. 

\subsection{Planar \Cg s of connectivity 1}\label{sec_Con1}


\begin{lem}\label{cg1con}
Every planar, locally finite, \Cg\ of connectivity 1 can be extended into a planar $2$-connected, locally finite, \Cg\ by adding redundant generators.
\end{lem}
\begin{proof}
We proceed by induction on the number of blocks incident with the vertex~$o$, where a \emph{block} means a maximal 2-connected subgraph in this subsection. Pick two such blocks $B,C$, an edge from $B$ corresponding to some generator $b$, and  an edge from $C$ corresponding to some generator $c$. Introduce a new redundant generator $x$ and the relation $x= b^{-1}c$.  Clearly, the resulting \Cg\ $G'$ obtained from the original \Cg\ $G$ by adding the generator $x$ has less blocks incident with~$o$ than~$G$.  

We claim that $G'$ is still planar. If none of $b^2$ or $c^2$ is a relator, then this is an easy exercise, based on the observation that $G$ can be embedded in such a way that for every vertex $v$, the edges labelled $b$ and $c$ emanating from $v$ lie in a common face boundary. 

If however $b^2$, say, is a relator, then it is a bit harder to avoid that the two $x$ edges emanating out of $o$ and $ob$ cross in our embedding. Still, the following observation will help us embed $G'$ in this case (and it is also applicable to the case where none of $b^2$ or $c^2$ is a relator). A good example to bear in mind throughout the rest of the proof is where $G$ is the \Cg\ $Cay\left<b,c\mid b^2,c^2\right>$ of the free product of two copies of $\Z/2\Z$, and $x=bc$.

\labtequ{attach}{Let $H_0$ be the graph consisting of a single vertex, and suppose that for every $i\in \N$, the graph $H_i$ is obtained from $H_{i-1}$ by attaching a planar graph $P_i$ to $H_{i-1}$ by identifying some vertex $p_i\in V(P_i)$ with some vertex $h_i\in V(H_{i-1})$, and possibly joining a neighbour $p'_i$ of $p_i$ to a neighbour $h'_i$ of $h_i$ with an edge. Then $\bigcup_{i\geq 0} H_i$ is planar.}

To prove this, we first use induction to show that $H_i$ is planar: given an embedding of $H_{i-1}$, observe that $p'_i,p_i$ lie in a common face $F_i$ since they are neighbours. Likewise, $h'_i, h_i$ lie in a common face of $P_i$, and we may assume that that face is the outer face by embedding $P_i$ appropriately. We now embed $H_i$ by drawing $P_i$ inside $F_i$ and, if there is a $h'_i-p'_i$ edge in $H_i$, joining $h'_i$ to $p'_i$ with an arc in $F_i$ that avoids the rest of the graph. 

The fact that $\bigcup_{i\geq 0} H_i$ is planar now follows from a standard compactness argument. 

\medskip
To complete our proof, we will show that our $G'$ can be constructed as described in \eqref{attach}. 

Indeed, let $\HF$ be the set of blocks (i.e.\ maximal 2-connected subgraphs) of~$G$, and let $H_1,H_2,\ldots$ be an enumeration of $\HF$ such that for $i>1$, $H_i$ is incident with some $H_j$ for $j<i$. Then $G'$ has the claimed structure, with the $x$-edges playing the role of the $h'_i-p'_i$ edges.
%
%
\end{proof}



\subsection{\Cg s of connectivity 2}\label{sec_Con2}

In this section, we will complete the proof of our main theorem by showing that every locally finite $2$-connected planar \Cg s admits a \gcplpr.

A \emph{cut} in a graph~$G$ is a set of vertices $C$ spanning a connected subgraph of~\G, \st\ the \emph{boundary}
\[
\partial C:=\{x\in V(G)\sm C\mid x\text{ has a neighbour in }C\}
\]
of $C$ is finite and $C\cup\rand C\neq V(G)$. The \emph{order} of $C$ is the cardinality of $\partial C$.

We call two cuts $C,D$ \emph{nested} if, setting $C^*:=V(G)\sm C$ and $D^*:=V(G)\sm D$, one of the four relations holds:
\[
C\sub D,\quad C\sub D^*,\quad C^*\sub D,\quad C^*\sub D^*.
\]
We call a set of cuts \emph{nested}, if every two of its elements are nested.

\begin{defi} \label{defblock}
Given a nested set $\CF$ of cuts, a \emph{block} is a maximal subgraph $H$ such that for every cut $C$, we have either $V(H)\sub C\cup \partial C$ or $V(H)\sub C^*$ but not both.
\end{defi}

To obtain a \emph{torso} of a block $H$ from~$H$ we add all edges $xy$ such that $\{x,y\}\sub V(H)$ is a boundary of a cut in~$\CF$.

Tutte~\cite{Tutte} showed that every finite $2$-connected graph $G$ has an $\Aut(G)$-invariant nested set~$\CF$ of cuts of order~$2$ whose torsos are either $3$-connected or cycles.
This theorem also holds for locally finite graphs, see Droms et al.~\cite{DSS-2Con}.
Nevertheless, we will refer to it as \emph{Tutte's theorem}.
To each such nested set of cuts, there is an associated tree~$T$ that admits a bijection from $V(T)$ to the blocks and boundaries of cuts in~$\CF$ such that, for any $t_1,t_2\in V(T)$ and any $t$ on the unique $t_1$--$t_2$ path in~$T$, the image of~$t$ separates the images of~$t_1$ and~$t_2$.\footnote{Readers that are familiar with tree-decompositions of graphs might notice that this just says that for every nested set of cuts, we find a tree-decomposition of the graph whose parts are the blocks and boundaries of cuts.}
We call this tree~$T$ the \emph{decomposition tree} of the set of cuts.

A \emph{$2$-separator} is the boundary of a cut of order~$2$.
Lemma~\ref{cg2con} allows us to assume that all 2-separators of $G$ are joined by an edge, i.e.\ they are hinges in the sense of Section~\ref{secPlprGen}.
Given two \Cg s $G,H$, we call $G$ a \emph{Tietze-supergraph} of~$H$ if there are presentations $\left<\SF_G\mid \RF_G\right>$ of~$\Gamma(G)$ and $\left<\SF_H\mid \RF_H\right>$ of~$\Gamma(H)$ with $G=Cay\left<\SF_G\mid\RF_G\right>$ and $H=Cay\left<\SF_H\mid\RF_H\right>$ and with $\SF_G\supseteq\SF_H$ and $\RF_G\supseteq\RF_H$.

\begin{lem}\label{cg2con}
Every planar $2$-connected \Cg\ $G$ has a planar Tietze-super\-graph~$H$ in which every pair of vertices that separates $H$ is connected by an edge. In addition, the new edges are labelled by a new redundant generator. (Moreover, if \G\ is locally finite, then so is $H$.)
\end{lem}
\begin{proof}
To begin with, pick a $\Gamma(G)$-invariant nested set $\CF$ of cuts of order~$2$.
This set exists due to Tutte's theorem mentioned above.
For every pair of non-adjacent vertices $x,y$ such that one component of $G-\{x,y\}$ lies in~$\CF$, we add a new redundant generator $a$ and relation $a= x^{-1}y$.
Let us show that the nestedness of~$\CF$ implies that we do not lose planarity.

Note that every $2$-separator lies on the boundary of some face.
So if we join $x_1$ and~$y_1$ by a new edge and also want to join $x_2$ and~$y_2$, then the only reason why we cannot do this is because the edge $x_1y_1$ separates the face on whose boundary the vertices $x_2$ and~$y_2$ lie.
So, originally, all four vertices $x_1,x_2,y_1,y_2$ are distinct and lie on a boundary $C$ of some face~$F$ in this order (either clockwise or anticlockwise).
For $i=1,2$, let $P_i$ be an $x_i$--$y_i$ path whose inner vertices lie in a component of $G-\{x_i,y_i\}$ that avoids $x_j$ and~$y_j$ for $j\neq i$.
As the two paths $P_i$ lie outside of~$F$, the path $P_2$ connects a vertex in the inner face of $P_1+y_1x_1$ to one in its outer face, which is impossible due to the Jordan curve theorem.
This proves that we can indeed add the aforementioned redundant generators and relations without losing planarity.

Since every vertex has only finitely many neighbours and every two of them can be separated by only finitely many $2$-separators (see e.g.\ \cite[Proposition 4.2]{ThomassenWoess}), the resulting \Cg\ $G'$ is still locally finite.
\end{proof}

\newcommand{\insep}{well-sep\-ar\-ated}
Call a graph \emph{\insep} if it is $2$-connected and every 2-separator is joined by an edge. 

\begin{thm}\label{thm_Con2Back}
Every planar,  locally finite, \insep\ \Cg~$G$ with $\kappa(G)=2$ admits a \gcplpr.
\end{thm}
\begin{proof}
Let $\CF$ be a $\Gamma(G)$-invariant nested set of cuts of order~$2$ as in Tutte's Theorem.
Let $\BF_o$ be the set of blocks (in the sense of Definition~\ref{defblock}) that contain the vertex~$o$.
For $B\in\BF_o$, let $S_B$ be the set of those generators~$s\in \SF\cup\SF\inv$ such that the edge with label $s$ starting at~$o$ lies in~$B$.
Then $\SF\cup\SF\inv$ is covered by the set of $S_B$.
We fix an embedding $\rho$ of \g in $\R^2$, and endow every $S_B$ with the cyclic order induced by $\rho$ at~$o$.
Let $\BF'_o\sub \BF_o$ be maximal such that no two distinct $B,B'\in \BF_o$ are of the form $B=g(B')$ for any $g\in\Gamma(G)$.
We can apply Theorem~\ref{thm_fundgr} to  each $B\in\BF'_o$ to obtain a set $\DF_B\sub \pi_1(B)$ that generates $\pi_1(B)$, and such that $\DF_B^\circ$ is a nested set of indecomposable closed walks that is invariant under the stabiliser of~$B$ in~$\Gamma(G)$.
Then it is easy to see that
\[
\DF:=\bigcup_{\substack{B\in\BF'_o\\ g\in\Gamma(G)}}g(\DF_B)
\]
generates $\pi_1(G)$.
Let $\RF_\DF$ be the set of words corresponding to closed walks in~$\DF^\circ$.
Easily, $\left<\SF\mid\RF_\DF\right>$ is a presentation of~$\Gamma(G)$.
Once more, we use Tietze-transformations to obtain a finite subset $\RF\sub\RF_\DF$ with $\left<\SF\mid\RF_\DF\right>=\left<\SF\mid\RF\right>$, which is possible as $\Gamma(G)$ is finitely presented (Droms~\cite[Theorem~5.1]{DroInf}).
To see that the set $\CF:=\{B_1,\ldots,B_n\}:=\{S_B\mid B\in\BF_o\}$ is a spin structure of~$\PF:=\left<\SF\mid\RF\right>$, it remains to show that the graph $\TF:=(\CF\cup\SF',\EF)$, where $xy\in\EF$ if and only if $x\in y$ or $y\in x$, is a tree.

Let us suppose that \TF\ is not a tree.
Obviously, \TF\ is connected.
So it contains some cycle $S_1s_1\ldots S_ms_mS_1$ with $S_i\in\CF$ and $s_i\in\SF'$.
For each $i\leq m$, let $B(S_i)\in\BF_o$ be such that $S_i=S_{B(S_i)}$.
As each element of $\BF_o$ is a block, there is some path $P_i$ in~$S_i$ connecting the end vertices of $s_{i-1}$ and $s_i$ distinct from~$o$ (with $s_0=s_m$).
The concatenation of all these paths $P_i$ is a cycle $C$ in~$G$ that crosses all hinges~$s_i$ precisely once as $S_i\neq S_{i+1}$ (with $S_{m+1}=S_1$).
But this is not possible as each cycle, and hence also~$C$, must lie in a unique block of~$G$.

For  $i\leq n$, let $B(i)$ be that element of~$\BF_o$ with $S_{B(i)}=B_i$.
For every hinge $b\in\SF$ incident with~$o$ and every $i\leq n$ with $b\in B_i$, let $\mu(b,i)$ be that $B_j$ with $b(B(i))=B(j)$.
So we have $b\inv\in B_j$.
Let $\sigma(i)$ be the spin of $B_i$ at~$o$.
To define whether every generator is spin-preserving or spin-reversing in each element of the spin-structure (it participates in), we remember that the blocks ---being either $3$-connected or cycles--- have a unique embedding in the plane.
So for $s\in\SF$ and $i\leq n$, we define $\tau(s,i)$ to be $0$ if $s$ is spin-preserving in~$B(i)$ and~$1$ otherwise.
(Note that $\tau$ is also defined if $s\notin B_i$.)
Clearly, $(\PF,\CF,\sigma,\mu,\tau)$ is a \gempr.

As every element of~$\DF^\circ$ lies in a unique block, every $R\in\RF$ is blocked with respect to~$\CF$ by definition, and the number of spin-reversing generators in~$R$ is even.
As $\DF$ is nested, it is easy to check that no two relators cross.
The fact that no cycle is a subgraph of any other cycle implies that no relator is a sub-word of a rotation of another relator, and hence our \gempr\ is a \gcplpr.
\end{proof}

With an argument similar to the proof of \cite[Corollary 3.4]{planarpresI}, we obtain:

\begin{cor}
Every planar \insep\ \Cg~\G\ with $\kappa(G)=2$ is the $1$-skeleton of an almost planar Cayley complex of~$\Gamma(G)$.
\end{cor}

\begin{proof}
Since $G$ is planar, there is an embedding $\rho'\colon G \to \mathbb R ^2$ by definition. We will extend $\rho'$ to the desired map $\rho$ from the Cayley complex $X$ of $\Gamma(G)$ with respect to the presentation $\left<\SF\mid\RF\right>$ from above. For this, given any 2-cell $Y$ of~$X$ with boundary cycle $C$, we embed $Y$ in the finite component of $\mathbb R^2 \setminus C$. It is a straightforward consequence of the nestedness of $\DF$ that the resulting map $\rho$ has the desired property.
\end{proof}

\subsection{Consistent embeddings lead to \splpr s}

In the previous section, we have seen that $2$-connected planar \Cg s admit \gcplpr s. However, if the \Cg\ has a consistent embedding, we obtain a bit more  even for $1$-connected graphs:

\begin{thm} \label{thmcons}
Every planar \Cg\ with a consistent embedding admits a \splpr.
\end{thm}

\begin{proof}
Let \g be such a graph.
First note that, by repeating the arguments of the proof of Lemma~\ref{cg2con}, we can join the two vertices of any $2$-separator $\{x,y\}$ by a new edge whenever $xy\notin E(G)$ and $G-\{x,y\}$ has two components $C$ with $\partial C=\{x,y\}$, while keeping the embedding consistent.
So we may assume that every maximal $2$-connected subgraph of~$G$ is \insep.

Let $\BF$ be a set of blocks of the maximal $2$-connected subgraphs of~$G$ consisting of one block from each $\Gamma(G)$-orbit.
As before, Theorem~\ref{thm_fundgr} gives us for each $B\in\BF$ a set $\DF_B$ that generates $\pi_1(B)$ such that $\DF_B^\circ$ is a nested set of indecomposable closed walks that is invariant under the stabiliser in $\Gamma(G)$ of~$B$.
Let $\RF_B$ be the set of words corresponding to the elements of~$\DF_B^\circ$.
As above, 
 Tietze-transformations give us a finite $\RF\sub\bigcup_{B\in\BF}\RF_B$ such that $\PF=\langle\SF\mid\RF\rangle$ is a finite presentation of~$\Gamma(G)$, where $\SF$ is the generating set of~$G$.

If we let $\sigma$ be the spin of one fixed vertex~$x$ and $\tau(s)=0$ if the edge from $x$ labelled~$s$ is spin-preserving and $\tau(s)=1$ otherwise, then $(\PF,\sigma,\tau)$ is a \splpr\ of~$\Gamma(G)$.
Indeed, nestedness of the closed walks in~$\DF_B^\circ$ implies that the corresponding words are non-crossing, the fact that they are indecomposable implies that no relator is a subword of any other relator, and the embedding implies that every relator contains an even number of spin-reversing letters.
\end{proof}

\section{Conclusions}\label{secConc}

We now put the above results together to prove the statements of the introduction.
Because of the redundant generators used in Lemmas~\ref{cg1con} and~\ref{cg2con}, we need to generalise our notion of \plpr\ slightly. We say that $s\in \SF$ is an {\em obviously redundant} generator of a presentation $  \left<\SF\mid \RF\right>$, if there is exactly one relator $W_s\in \RF$ in which $s$ appears, and $s$  appears exactly once in~$W$. 
A {\em general \plpr} is a presentation obtained from a generic \plpr\ by recursively removing zero or more obviously redundant generators $s$ along with the corresponding relator $W_s$. The last two sections prove the two directions of \Tr{mainthm}:

\begin{proof}[Proof of \Tr{mainthm}]
If \g is a finitely generated planar \Cg, then by Lemmas~\ref{cg1con} and~\ref{cg2con} we may find a Tietze-supergraph that is is 2-connected and \insep.  \Tr{thm_Con2Back} then yields a \gcplpr, from which we can remove any generators that were not present in \g to obtain a \glplpr\ of \G, which proves the forward direction.

For the backward direction, if \g admits a \glplpr, then some supergraph $G'$ admits a \gcplpr, and is thus planar by \Tr{Gplanar}. Since planarity is preserved under deleting edges, so is \G.
\end{proof}

A similar result holds when we insist that there is a consistent embedding, and we can even allow our \Cg s to have infinitely many generators:
\begin{thm}\label{iffthmcon}
A Cayley graph  admits a consistent embedding in the plane \iff\ it admits a \splpr.
\end{thm}
The two directions of \Tr{iffthmcon} are given by Theorem~\ref{thmcons} and Theorem~\ref{thmplanar}.

\medskip
Next, we use our presentations to obtain effective enumerations.

\begin{thm}\label{thm_maincon}
The \Cg s that admit a consistent embedding in the plane are effectively enumerable.
\end{thm}
\begin{proof}
By \Tr{iffthmcon}, it suffices to produce an effective enumeration of the \splpr s. 
For this, it suffices to produce an enumeration of the embedded presentations, and output those embedded presentations that satisfy the three conditions in the definition of a \splpr\ (Definition~\ref{defsplpr}); it is easy to see that these conditions can be checked algorithmically.
\end{proof}

\begin{thm}\label{thm_maincon2}
The planar, locally finite \Cg s are effectively enumerable.
\end{thm}
\begin{proof}
Similarly to the proof of Theorem~\ref{thm_maincon}, we remark that any effective enumeration of the \glplpr s gives rise to an effective enumeration of the planar \Cg s by \Tr{mainthm}.

To effectively enumerate the \glplpr s, we start with an enumeration of the generic embedded presentations, and output those that satisfy the four conditions of Definition~\ref{defgplpr}, which can be checked algorithmically. Having thus effectively enumerated the \gcplpr s, we remove any obviously redundant generators to effectively enumerate the \glplpr s: for each output $G= \left<\SF\mid \RF\right>$, check for every $s\in \SF$ whether $s$ is an {obviously redundant} generator. For every such $s$ found, output the presentation $G':= \left<\SF\sm \{s\} \mid \RF \sm \{W_s\}\right>$. Then, recursively apply the same check to~$G'$, removing any obviously redundant generators of that presentation and so on.
\end{proof}

We conclude with some related questions concerning embeddings of Cayley complexes.
Let $CC(\PF)$ denote the Cayley complex of a presentation $\PF$. Call a map $\rho\colon CC(\PF) \to \R^2$ \emph{\cns} if its restriction to $Cay(\PF)$ is \cns. Call $\rho$ \emph{nested} if it witnesses the fact that $CC(\PF)$ is almost planar, i.e.\ if 
the images under $\rho$ of the interiors of any two 2-cells are either disjoint, or one is contained in the other.

The following might be interesting as it exhibits a geometric property of Cayley complexes which can be decided by an algorithm.

\begin{thm}\label{thm_decidS}
There is an algorithm that given a presentation $\PF=\left< \SF \mid \mathcal \RF \right>$ decides whether $CC(\PF)$ admits a nested, consistent map into $\R^2$.
\end{thm}
\begin{proof}
We claim that $CC(\PF)$ admits a nested, consistent map into $\R^2$ \iff\ there is a spin $\sig$ on $\SF$ and a `spin-behaviour' function $\tau$ from $\SF$ to $\{0,1\}$ \st\ the triple $(\PF, \sigma, \tau) $ is a \splpr.

To prove the backward direction, note that if $\PF, \sigma, \tau $ is a \splpr, then $Cay(\PF)$ admits a consistent embedding $\rho$ into $\R^2$ by \Tr{thmplanar}. Extend this embedding into a map $\rho'$ from $CC(\PF)$ to $\R^2$ by mapping each 2-cell inside the closed curve to which $\rho$ maps its boundary. Then $\rho'$ is nested because no two words in $\RF$ cross each other by the definition of a \splpr.

For the forward direction, given such a map $\rho\colon CC(\PF) \to \R^2$, we can read the spin data $\sigma, \tau $ from $\rho$ since $\rho$ is consistent. Then $\PF, \sigma, \tau $ is an \empr. To prove that it is a \splpr\ it remains to show that no two words in $\RF$ cross each other, which follows immediately from the nestedness of  $\rho$.
\end{proof}

By using \glplpr s instead of special ones, \Tr{thm_decidS} can be generalised to yield a further decidable property of Cayley complexes, but instead of maps into $\R^2$ we have to consider maps into larger spaces obtained by glueing copies of $\R^2$ along (possibly closed) bounded simple curves ---to which we map the hinges of our \Cg s--- in a tree like fashion. We leave the details to the interested reader.

Our results do not yet answer the following
\begin{problem}
Is there an algorithm that given a presentation $\PF=\left< \SF \mid \mathcal \RF \right>$ decides whether $CC'(\PF)$ is planar?
\end{problem}
In this problem $CC'(\PF)$ denotes the complex obtained from $CC(\PF)$ by removing redundant 2-cells, that is, if a set of 2-cells have the same boundary, we remove all but one of them. Some authors still call $CC'(\PF)$ the Cayley complex of $\PF$. (In \Tr{thm_decidS} it does not make a difference whether we consider $CC(\PF)$ or $CC'(\PF)$.)

We remark that it is not true that $CC(\PF)$ is planar \iff\ $\PF$ is a facial presentation in the sense of \cite{vapf}; the presentation $\PF= \left<a,b \mid a^2, b^3, ab^{-1}\right>$ if facial, but its Cayley complex consists of a single vertex, two loops, a 2-cell winding twice around a loop, and a 2-cell winding three times around the other loop.

\medskip
Having studied embeddings of Cayley complexes in $\R^2$, the following suggests itself
\begin{problem}
Which groups admit a Cayley complex embeddable in $\R^3$?
\end{problem}

\section{Further remarks} \label{secFR}

We proved that every planar \Cg\ \g admits a \plpr\ such that every relator induces a cycle of \g (rather than an arbitrary closed walk with repetitions of vertices). It would be interesting if we could strengthen the definition of a \plpr\ in such a way that this is always the case in the resulting planar \Cg. Some strengthening will be necessary as shown by the  example $\PF= \left<a,b \mid a^2, b^3, ab^{-1}\right>$ from the previous section. This is a \plpr\ ---even stronger, every relator is facial--- but it is easy to see that its group is the group of one element. Our optimism that this may be possible stems from the fact that it was possible in the cubic case~\cite{cayley3}.

If we could do this, then it would probably help to prove that the planar \Cg s are effectively constructible:
\begin{conj}
There is an algorithm that given a \gplpr~$\PF$, and $n\in \N$, outputs the ball of radius $n$ in the \Cg\ of $\PF$.
\end{conj}
This was proved in \cite{cayley3} in the cubic case.

\medskip
A further interesting question, also asked in \cite{cayley3}, is whether \fe\ $n\in \N$ there is an upper bound $f(n)$, \st\ every $n$-regular planar Cayley graph admits a planar presentation with at most $f(n)$ relators. This would strengthen Droms' result~\cite[Theorem~5.1]{DroInf} that planar groups are finitely presented.
 
\medskip
We conclude with a rather unrelated observation. 
It is known that the fundamental group of a finite graph of
groups with residually finite
vertex groups and finite edge groups is residually finite \cite[II.2.6.12]{SerArb}. Combining this with Dunwoody's result mentioned in the introduction, we obtain the following corollary, to which this paper has no contribution
\begin{cor} \label{corresf}
Every planar group is residually finite.
\end{cor}

\bibliographystyle{plain}
\bibliography{../collective}

\end{document}